\documentclass[intlimits]{amsart}

\usepackage[all,hyperref]{bi-discrete}
\usepackage[backend=biber,maxbibnames=5,maxalphanames=5,style=alphabetic,bibencoding=utf8,giveninits,url=false,isbn=false,sorting=nyt]{biblatex}
\usepackage{esint}
\usepackage{todonotes}  

\usepackage[
    centering      % This ensures margins are equal on left and right
]{geometry}

\providecommand{\eps}{{\mathcal E}}

\AtBeginBibliography{\small}

\usepackage{tikz}
\usepackage{pgfplots}

\allowdisplaybreaks
\pgfplotsset{width=10cm,compat=1.9}

\begin{document} 

\title[Maximal regularity for the symmetric gradient p-Laplace system]{Maximal Sobolev regularity of the stress tensor for the symmetric gradient p-Laplace system}%

\author{Linus Behn, Andrea Cianchi, Lars Diening, and Fa Peng}

\address{Linus Behn \\
	Fakult\"{a}t  f\"{u}r Mathematik,
	University of Bielefeld\\
	Universit\"{a}tsstrasse 25, 33615 Bielefeld, Germany}
\email{linus.behn@math.uni-bielefeld.de}

\address{Andrea Cianchi \\
	Dipartimento di Matematica e Informatica \lq\lq U.Dini"\\
    Universit\`{a} di Firenze,
	Viale Morgagni 67/a, 50134 Firenze, Italy}
\email{andrea.cianchi@unifi.it}

\address{Lars Diening \\
	Fakult\"{a}t  f\"{u}r Mathematik,
	University of Bielefeld\\
	Universit\"{a}tsstrasse 25, 33615 Bielefeld, Germany}
\email{lars.diening@uni-bielefeld.de}

\address{Fa Peng\\
	School of Mathematical Sciences,
	Beihang University\\
	Changping District Shahe Higher Education Park South Third Street No. 9, Beijing 102206, P.R. China}
\email{fapeng@buaa.edu.cn}


%\date{}%
%\dedicatory{}%
%\commby{}%

\begin{abstract}
The symmetric $p$-Laplace operator enters various   models in mathematical physics, such as incompressible materials with power-type hardening and non-Newtonian fluids. In this work,
second-order differentiability properties of solutions to the symmetric $p$-Laplace system are established. They are formulated as maximal Sobolev regularity of the nonlinear stress tensor for locally square integrable right-hand sides. 
\end{abstract}

\maketitle

%\tableofcontents

\section{Introduction and main results}\label{S:intro}

The symmetric gradient $p$-Laplace
system has the form
\begin{align}\label{eq:system}
    -\divergence (\mathcal A_p(\eps u))=f \quad \text{in $\Omega$,}
    \end{align}
 where $\Omega$ is an open set in  $\RRn$, with $n\geq 2$. Here, $u\colon\RRn \to \RRn$ is 
  the unknown and $\eps u\colon\RRn \to \mathbb R^{n\times n}$ denotes its \emph{symmetric gradient}, given by
\begin{align*}
    \eps u = \tfrac 12 (\nabla u + (\nabla u )^T),
\end{align*}
where the superscript $\lq\lq T"$ stands for transpose.
Moreover,  $1<p<\infty$,
the function $ \mathcal A_p : \mathbb R^{n\times n} \to \mathbb R^{n\times n}$ is defined  as
\begin{align}
     \label{Ap}
     \mathcal A_p(\xi)= |\xi|^{p-2}\xi \quad \text{for $\xi \in \mathbb R^{n\times n}$,}
\end{align}
and $f: \Omega \to \RRn$. The quantity 
$\mathcal A _p (\eps u)$ is usually referred to as the \emph{stress tensor} in the literature, and its maximal Sobolev regularity is the subject of this work.

Nonlinear systems of partial differential equations depending on the symmetric gradient of solutions, of which \eqref{eq:system}
is a prototypical instance, come into play in various mathematical models for physical phenomena, such as plasticity theory and  non-Newtonian fluid dynamics. In particular, the system \eqref{eq:system} models incompressible materials with power-type hardening,  where $u$ is the displacement field, $\eps u$ the infinitesimal strain tensor, and $\mathcal A _p (\eps u)$  the deviatoric stress tensor -- see, for instance, \cite{Seregin1992,SereginShilkin1997,FuchsSeregin2000}. In the context of non-Newtonian fluids, \eqref{eq:system} is closely related to the  $p$-Stokes system, which contains an additional unknown  pressure term and a divergence vanishing condition.  In the latter model, which is discussed e.g. in  \cite{MalekNecasRuzicka1993, Ladyzenskaja1967}, the quantity $\mathcal{A}_p (\eps u)$ is called the viscous stress tensor.

The system \eqref{eq:system} is a counterpart of the classical $p$-Laplace system, where the  symmetric gradient is replaced by the full gradient:
\begin{align}\label{classical-system}
    -\divergence (\mathcal A_p(\nabla u))=f \quad \text{in $\Omega$.}
\end{align}
Here,  $u: \Omega \to \mathbb R^N$, with $N\geq 1$. 

The regularity theory for \eqref{classical-system} and its  generalizations has been developed over more than seventy years, and nowadays it constitutes  the corpus of a vast literature. By contrast, much less is known about our system \eqref{eq:system} and similar  systems depending on the symmetric gradient.
For instance, the $C^{1,\alpha}$-regularity of solutions to the system \eqref{classical-system} was established in the 1970s, starting with the work of K.Uhlenbeck \cite{Uhlenbeck1977}, the case of  scalar solutions ($N=1$) having already been settled by N.N.Ural'ceva \cite{Uraltseva1968} in the preceding decade.
On the other hand, 
 it is a classical open problem whether the  solutions to \eqref{eq:system} are of class $C^{1,\alpha}$.  Such regularity has only been proved  in the case when $n=2$, see  \cite{KaplickyMalekSara1999,DieningKaplickySchwarzacher2014} where  the $p$-Stokes system is also considered. 
A major technical obstacle in this connection is that, unlike  $\abs{\nabla u}$ in the case of solutions to \eqref{classical-system}, the function $\abs{\eps u}$ is not a subsolution of a scalar elliptic equation. This prevents the application of a De Giorgi type iteration method to $\abs{\eps u}$. The recent developments of this method introduced in \cite{BehnDieningNowakScharle2025}  to directly handle the vectorial solution $u$ to systems involving the full gradient $\nabla u$ also do not seem adaptable to the case of symmetric gradient systems.

The focus of the present contribution is on local second-order regularity properties of solutions to \eqref{eq:system} in the spirit of the classical  linear theory for data $f\in L^2_{\rm loc}(\Omega)$. In the linear case, corresponding to $p=2$, the natural regularity is $u\in W^{2,2}_{\loc}(\Omega)$. When $p \neq 2$, the maximal transfer of regularity from $f$ to $u$ is a much more delicate issue.

For the standard $p$-Laplace system \eqref{classical-system}, this question has been addressed
 in the recent papers \cite{CianchiMazya18,CianchiMazya19,BalciDieningCianchiMazya22}, where it is shown that this kind of regularity is properly described in terms of the membership of the stress tensor $\mathcal A_p(\nabla u)$ in the Sobolev space $W^{1,2}_{\rm loc}(\Omega)$. 
The scalar case of \eqref{classical-system}, corresponding to $N=1$,   is the subject of \cite{CianchiMazya18}, where  it is shown that, for every $p>1$,
\begin{align}
    \label{maximal-classical}
    \mathcal A_p(\nabla u) \in W^{1,2}_{\rm loc}(\Omega) \quad \text{if and only if} \quad f \in L^2_{\rm loc}(\Omega).
\end{align}
A two-sided local estimate associated with \eqref{maximal-classical} is also offered.
In the vectorial case, i.e. when $N\geq 2$, an analogous result holds for $p>2(2-\sqrt 2) \approx 1.1715$. This is established in the paper \cite{BalciDieningCianchiMazya22}, which augments the previous contribution \cite{CianchiMazya19}, where the same conclusion was shown to hold in the more restricted range of values of $p>\tfrac 32=1.5$. Whether or not the property \eqref{maximal-classical} continues to hold for every $p>1$ also for systems  is still an open problem.  
 Various papers on this and related topics appeared in the wake of these results -- see e.g. \cite{ACCFM, ACF, ACP, BMV, GuMo, HaSa, KaSa, MiPeZh, MoMuSci, Sar}. However, most of them deal with scalar problems.

The Sobolev regularity of $\mathcal A_p(\eps u)$ for solutions to the symmetric gradient system \eqref{eq:system}
is a yet more  challenging problem due to the stronger coupling of the nonlinear system. Moreover, the result cannot hold for every $p>1$. Indeed, even 
in the special case when $f=0$, the stress tensor $\mathcal A_p(\eps u)$ need not belong to $W^{1,2}_{\rm loc}(\Omega)$ for small $p$, as the  example below shows. This raises the question of identifying the  range of values of $p$ for which the problem has a positive answer.

\begin{example}\label{ex1}
    The function $u:\RRn \to \RRn$, defined as
\begin{align*}
	 \qquad u(x)=(2x_1x_2,-x_1^2,0,\dots,0)^T \quad \text{for $x\in \RRn$,}
\end{align*}
satisfies 
\begin{equation*}
	\mathcal A_p(\eps u(x))=
	\begin{pmatrix}
		(2|x_2|)^{p-1}{\rm sign} (x_2)&0 & \dots & 0\\
		0&0 & \dots & 0
        \\ \vdots & \vdots & \ddots & \vdots
        \\ 0 & 0& \dots & 0
	\end{pmatrix}.
\end{equation*}
Hence, it solves the system 
$$-\divergence (\mathcal A_p(\eps u))=0.$$
On the other hand, because of the singularity of $u$ on the hyperplane $\lbrace x_2=0\rbrace$, one has that $$\text{$\mathcal A_p(\eps u) \in W^{1,2}_{\rm loc}(\RRn)$ \,\, if and only if \,\,$p>\tfrac32$.}$$
\end{example}

Our main result, contained in Theorem \ref{thm:local} below, exhibits a  substantial spectrum of values of $p >\frac 32$, that just excludes a  tiny range in low dimension, for which 
\begin{align}
    \label{maximal-symm}
    \mathcal A_p(\eps u) \in W^{1,2}_{\rm loc}(\Omega) \quad \text{if and only if} \quad f \in L^2_{\rm loc}(\Omega).
\end{align}
Moreover, it implies the two-sided estimate
\begin{align}\label{est-symm}
   c_1 \norm{\mathcal A_p(\eps u)}_{W^{1,2}( \Omega'')} \leq \norm{f}_{L^2(\Omega')}+\|\mathcal A_p(\eps u)\|_{L^{1}(\Omega')}\leq  c_2\norm{\mathcal A_p(\eps u)}_{W^{1,2}(\Omega')}
\end{align}
for every open sets $ \Omega '' \subset\subset \Omega' \subset\subset \Omega$ and some positive constants $c_1$ and $c_2$. Plainly, only the first inequality in \eqref{est-symm}
is nontrivial.
 Statements of this kind are often referred to as \emph{maximal regularity results}, in that they yield an isomorphism between the  classes of the datum $f$ and  of a quantity depending on the solution under consideration, which in the current situation  is the stress tensor~$\mathcal{A}_p (\eps u)$. 

The admissible values of $p$ in Theorem \ref{thm:local} for \eqref{maximal-symm} and \eqref{est-symm} to hold  line in the interval $(p^-(n), p^+(n))$, where
\begin{align}
    \label{ASSUMPTIONS}
      {p^-(n)} = \begin{cases}
  2- \frac{5}{2(4+\sqrt{6})} &\quad \text{if $n=2$}
   \\ 2- \frac{1}{\sqrt{n+1}+1}  &\quad \text{if $n\geq 3$}
\end{cases}
\qquad \text{and} \qquad 
{p^+(n)} = \begin{cases}
     \infty &\quad \text{if $n\leq 7$}
   \\ 2+ \frac{1}{\sqrt{n+1}-1} &\quad \text{if $n\geq 8$.}  
   \end{cases}
\end{align}

For dimensions $1,\dots ,8$ these values   are collected in Table \ref{table}.
In light of Example \ref{ex1}, it shows that, for $2\leq n\leq 7$,  the question of the validity of \eqref{maximal-symm}
remains open only for $p$ in a  small interval above $\tfrac 32$.

\begin{table}[htbp]
\centering
\renewcommand{\arraystretch}{1.5}
\begin{tabular}{|c|c|c|c|c|c|c|c|}
\hline
$n$ & $2$ & $3$ & $4$ & $5$ & $6$ & $7$ & $8$  \\
\hline&&&&&&&\\[-3.9ex]\hline
$p^+(n)$ & $\infty$ & $\infty$ & $\infty$ & $\infty$ & $\infty$ & $\infty$ & $2.5$ \\
\hline
$p^-(n)$ & $\approx 1.612$ & $\approx 1.667$ & $\approx 1.691$ & $\approx 1.710$ & $\approx 1.726$ & $\approx 1.739$ & $1.75$ \\
\hline
\end{tabular}%
\medskip
\caption{Values of $p^-(n)$ and $p^+(n)$.
}
\label{table}
\end{table}

Theorem \ref{thm:local} applies to local weak solutions to the system \eqref{eq:system} whenever $f\in (W^{1,p}_0(\Omega'))^*$ for every open set $\Omega'\subset \subset \Omega$.
Since we are merely assuming that  $f \in L^2_{\loc}(\Omega)$,   this is the case only if $p\geq\frac{2n}{n+2}$. In high dimensions ($n\geq 21$), the lower bound on $p$ required in our result falls below this threshold. In this regime,  solutions have to be interpreted according to a more 
general, weaker notion involving a limiting process via approximations of $f$. They will be called approximable solutions. Their definition is inspired by the methods of \cite{DolzHungerMueller97} and is introduced in Section \ref{S:functional}. Here, we content ourselves by mentioning   
that an approximable solution   is only approximately differentiable in the geometric measure sense. For simplicity of notation, the symmetric part of the approximate gradient of a function $u$ will still be denoted by $\eps  u$ throughout. 

The definition of approximable solution  actually extends that of weak solution. Indeed, as shown in Section \ref{S:solutions},  
 any weak solution is also an approximable solution. For this reason,  our result is just stated for solutions of the latter kind. In the statement, and in what follows, $B_R$ denotes a ball in $\RRn$ of radius $R$.

\begin{theorem}\label{thm:local}
Assume that $p^-(n) <p< p^+(n)$.
Let $f \in L^2_{\loc}(\Omega)$ and let $u$ be a local approximable solution to the system \eqref{eq:system}.
Then,
\begin{align}
    \label{main0}
    \mathcal A_p(\eps u)
\in W^{1,2}_{\rm loc}(\Omega),
\end{align}
 and there exists a constant $c=c(p,n)$ such that
\begin{align}\label{main1}
 R^{-1}\|\mathcal A_p(\eps u)\|_{L^2(B_R)}  + \,\|\nabla (\mathcal A_p(\eps u))\|_{L^2(B_R)}
   \leq c \big(\,\|{f}\|_{L^2(B_{2R})} + R^{-\frac n2-1}\|\mathcal A_p(\eps u)\|_{L^{1}(B_{2R})}\big)
\end{align} 
 for every ball $B_R$ such that $B_{2R}\subset\subset \Omega$.
\end{theorem}

As far as we are aware,   Theorem \ref{thm:local} is the first contribution on the regularity of solutions to the system \eqref{eq:system} expressed in terms of the stress tensor $\mathcal A_p(\eps u)$. As already pointed out  in connection with Equations \eqref{maximal-symm} and \eqref{est-symm}, this yields maximal regularity properties in the sense specified above. Prior results  about the regularity of the more customary quantity $\mathcal V_p(\eps u)= |\eps u|^{\frac{p-2}2}\eps u$,  called the energetic stress tensor, are available in the literature.  They can be found, e.g., in \cite{SereginShilkin1997,BerselliRuzicka2020,BehnDiening24,GiPa,CianchiGianettiPassarelliScheven2026, BeKaRu, BerselliRuzicka2020, Seregin1992, SeSh,Gmeineder2020}, and,  for solutions to the $p$-Stokes system, in  \cite{Breit2012, BreitFuchs, DieningKaplicky13, FuchsSeregin2000, GiPaSch1, GiPaSch2, Naumann1988}. However, as shown by these  results, the use of $\mathcal V_p(\eps u)$ does not allow for sharp conclusions reflecting the regularity of the datum $f$.

 This perspective mirrors a paradigmatic parallel shift  from $\mathcal{V}_p(\nabla u)$ to $\mathcal{A}_p(\nabla u)$ in the description of the regularity of solutions to the $p$-Laplace equation  and system over the last two decades. For instance, besides the series of papers \cite{CianchiMazya18,CianchiMazya19,BalciDieningCianchiMazya22} on second-order regularity mentioned above, the expression  $\mathcal{A}_p (\nabla u)$ is employed   in  \cite{DieningKaplickySchwarzacher2012} in the context of $\setBMO$ regularity, and in \cite{BCDKS, BreitCianchiDieningSchwarzacherARMA} for  regularity in more general Campanato type spaces. These contributions deal with data on the right-hand side in divergence form.
  The norm of the same expression is   pointwise bounded via Riesz potentials of $f$ in \cite{KuusiMingione2013,KuusiMingione2018}. For solutions to boundary value problems, an estimate for $\mathcal{A}_p (\nabla u)$ in rearrangement form in presented in \cite{CianchiMazyaJEMS}. Fractional regularity of $\mathcal{A}_p (\nabla u)$ is the subject of \cite{AvelinKuusiMingione} and \cite{BalciDieningWeimar}.

Our approach to Theorem \ref{thm:local} is based on pointwise differential identities and ensuing inequalities for derivatives of the stress tensor for a class of systems, governed by an Orlicz type nonlinearity, which includes \eqref{eq:system} as a special instance.
 These inequalities constitute the core of the present paper. Let us emphasize that two different kinds of identities and inequalities are needed to derive the best possible results in low and high dimension $n$. The relevant identities require a degree of smoothness of functions which is not guaranteed for solutions to \eqref{eq:system}. This calls for approximations at various levels, including a regularization of the differential operator in \eqref{eq:system} via nonlinear operators with Orlicz quadratic growth, and explains the necessity of a detour into the Orlicz realm. Handling solutions to the regularized systems is still not straightforward, because of their limited regularity properties due to the dependence of the differential operator on the symmetric gradient.

 The conclusions of Theorem \ref{thm:local} continue to hold for a slightly broader class of systems of the form
\begin{align}\label{eq:systembis}
    -\divergence (\mathcal A_p^\nu(\eps u))=f \quad \text{in $\Omega$,}
    \end{align}
    where $\nu\geq 0$ and the function $A_p^\nu$ is given, according to the Carreau law \cite{Carreau1972}, by
    \begin{align}\label{Apnu}
        A_p^\nu (\xi)= (\nu+|\xi|^2)^{\frac{p-2}{2}}\xi \quad \text{for $\xi \in \mathbb R^{n\times n}$.}
    \end{align}
    Of course, the function $\mathcal A_p$ has to be replaced with $A_p^\nu$ in \eqref{main0} and \eqref{main1}. The same proof of Theorem \ref{thm:local} applies, with minor modifications. Actually, if $\nu >0$, the differential operator in \eqref{eq:systembis} is more regular than that in \eqref{eq:system}, since it is neither degenerate nor singular where $\eps u=0$.
    \\ Symmetric gradient systems with an even more general, non-necessarily polynomial growth can also be dealt with via the methods employed in this paper. The Orlicz-Sobolev spaces provide the appropriate functional framework for their solutions. An extension of Theorem \ref{thm:local} that includes these systems
requires assuming suitable lower and upper bounds on the Simonenko indices, as defined in Section \ref{sec:young}, of the function which replaces $t^{p-2}$. In fact, as explained above, 
     our pointwise differential inequalities are already proved in this  setting.
     In order to avoid unnecessary technical complications, we limit our discussion to the system \eqref{eq:system}. The  variants needed to treat these generalized versions are briefly described in Remarks \ref{rem:Orlicz1} and  \ref{rem:orlicz2}.

The remaining part of this paper is organized as follows. In Section \ref{sec:young} we collect some preliminary material concerning Young functions associated with the differential operator of the system \eqref{eq:system} and with its regularizations. The key pointwise estimates involving the symmetric gradient are the subject of Section \ref{sec:pointwise}. 
Section \ref{S:functional} is devoted to the function space framework and Section \ref{S:solutions} to the notions of  solutions entering our analysis.  The Sobolev regularity of the stress tensor for nonlinear systems with quadratic growth, which plays a role in the approximation process, is established in Section \ref{sec:quadratic}. Finally, the proof of Theorem \ref{thm:local} is accomplished 
 in Section \ref{S:proof-main}.

\section{\texorpdfstring{Young functions associated with the symmetric gradient $p$-Laplace system}{Young functions associated with the symmetric gradient p-Laplace system}}\label{sec:young}

This section is devoted to definitions and basic properties of certain functions which play a role in our analysis of the system \eqref{eq:system}. In what follows, the relation $\lq\lq \lesssim "$ between two expressions means that they are bounded by each other, up to positive multiplicative constants depending on quantities to be specified. The relations $\lq\lq \gtrsim"$ and $\lq\lq \eqsim "$ are defined accordingly.

Given $p>1$, denote by $a_p: (0, \infty) \to (0, \infty)$ the function defined as
$$a_p(t)=t^{p-2} \quad \text{for $t>0$.}$$ 
Thus, the function $\mathcal A_p$ given by \eqref{Ap} takes the form 
$$\mathcal A_p(\xi)= a_p(|\xi|)\xi \quad \text{for} \,\,\xi \in  \RRn,$$ and the system \eqref{eq:system} can be rewritten as
\begin{align}
  \label{eq:eff-viscosity}
  -\divergence\big( a_p(\abs{\eps u}) \eps u\big) &= f.
\end{align}
In the context of power-law plasticity, the factor $a_p(\abs{\eps u})$ is called the \emph{effective modulus}, whereas in the context of Non-Newtonian fluids it is called the \emph{effective viscosity}.
\\ Notice that, on setting 
$$b_p(t)=a_p(t)t \quad \text{and} \quad B_p(t) = \int_0^t  b_p(\tau) \,d\tau \quad \text{for $t>0$,}$$
one has that $b_p(t)=t^{p-1}$ and $B_p(t)=\tfrac 1p t^p$, and hence
\eqref{eq:system} is the Euler-Lagrange system of the energy functional
\begin{align}
\label{eq:energy}
    \mathcal{J}_p(u) = \int_\Omega B_p(\abs{\eps u})\,dx - \int_\Omega u \cdot f \,dx.
\end{align}
 Regularizations of the functions $a_p$, $b_p$ and $B_p$ are needed in our approximation process of the solutions to \eqref{eq:system}. This requires enlarging the class of functions under consideration from powers to the more general family of 
Young functions. 

Given a function $a: (0, \infty) \to  (0, \infty)$ such that $a\in C^1(0,\infty)$, we define $\theta_a: (0,\infty) \to \setR$ as
    \begin{align}
        \label{theta}
        \theta_a(t)=\frac{a'(t)t}{a(t)} \qquad \text{for $t>0$}.
    \end{align}
The lower and upper Simonenko indices of $a$ are then given by
 \begin{align}\label{ia}
i_a= \inf _{t >0} \theta_a (t) \qquad \hbox{and} \qquad
s_a= \sup _{t >0} \theta_a (t).
\end{align}
 Thus,  $i_a\leq \theta_a (t)\leq s_a$ for $t>0$. In particular, 
$\theta_{a_p}(t)=i_a=s_a=p-2$ for $t>0$.
\\
In what follows, we shall always assume that
\begin{align}\label{infsup}
-1 < i_a \leq s_a < \infty.
\end{align}
Thanks to the first inequality in \eqref{infsup}, the function $b : [0,
\infty ) \to [0, \infty )$,  defined as
\begin{align}
    \label{b} 
    b(t)= \begin{cases}
        ta(t)  &\quad  \hbox{if $t > 0$}
        \\ 
        0  &\quad  \hbox{if $t = 0$,}
    \end{cases}
\end{align}
 is continuous in $[0,\infty)$ and strictly increasing. Hence, the
function $B: [0, \infty ) \to [0, \infty )$,  given by
\begin{align}
    \label{B}
    B(t) = \int _0^t b (\tau ) \, d\tau \qquad \text{for $t \geq 0$,}
\end{align}
is a finite-valued Young function. Namely, $B$ is convex, continuous, and $B(0)=0$.  Note that $B$ is a uniformly convex Young function in the sense of~\cite[Appendix~B]{DieningFornasierWank20} with lower and upper index of uniform convexity 
$i_a+2$ and $s_a+2$, respectively. By~\eqref{infsup}, we have $1<i_a+2\le s_a+2 < \infty$.

The Young conjugate $\widetilde B$ of $B$ can be defined as the Young function obeying
\begin{align}
    \label{Btilde}
    \widetilde B(t) = \int _0^t b^{-1} (\tau ) \, d\tau \qquad \text{for $t \geq 0$.}
\end{align}
The following version of Young's inequality holds:
\begin{align}
    \label{youngineq}
    st \leq B(s) + \widetilde B(t) \quad \text{for $s,t \geq 0$.}
\end{align}
We also define the function $\mathcal A : \mathbb R^{n\times n} \to \mathbb R^{n\times n}$ as
\begin{align}
    \label{mathcalA}
\mathcal A (\xi) = \begin{cases}
    a(|\xi|)\xi & \qquad \text{if $\xi \neq 0$}
    \\ 0 &\qquad \text{if $\xi =0$.}
\end{cases}
\end{align}
One can verify  that
\begin{equation}\label{jan200}
\text{$i_b= i_a+1$\quad  and \quad $s_b=s_a+1$,}
\end{equation}
and
\begin{equation}\label{jan201}
\text{$i_B\geq i_b+1$ \quad and \quad  $s_B\leq s_b +1$.}
\end{equation}
Moreover, since the function $b$ is increasing,
\begin{align}
    \label{feb120}
    B(t) \leq b(t)t \quad \text{for $t \geq 0$.}
\end{align}
From the definition of $i_B$ and $s_B$ one can deduce that
\begin{equation}
    \label{mar0}
    B(1) \min\{t^{i_B}, t^{s_B}
    \}  \leq B(t) \leq B(1) \max\{t^{i_B}, t^{s_B}
    \}  \quad \text{for $t \geq 0$.}
\end{equation}
More generally, one has that
\begin{align}
    \label{mar00}
    B(t) \min\{\lambda ^{i_B}, \lambda ^{s_B}
    \}  \leq B(\lambda t) \leq B(t) \max\{\lambda ^{i_B}, \lambda^{s_B}
    \}  \quad \text{for $t, \lambda  \geq 0$.}
\end{align}
As a consequence, both  $B$ and $\widetilde B$ satisfy the $\Delta_2$-condition, 
 and  
\begin{align}
    \label{orliczbasic}
 c_1 B(t)  \leq  \widetilde{B}(b(t)) \leq c_2 B(t) \quad \text{for $t\geq 0$,}
\end{align}
for some positive
 constants $c_1$ and $c_2$ depending only on $i_B$ and $s_B$. 
\\
Under the additional assumption that
\begin{align}
        \label{dec1}
        b\in C^1([0, \infty)),
    \end{align}
 the function $a$ can be extended in a continuous way at $0$, by setting
 \begin{align}
     \label{a(0)}
     a(0)=b'(0),
 \end{align}since
\begin{align}
    \label{dec2}
\lim_{t\to 0^+}a(t) =\lim_{t\to 0^+}\frac{b(t)}t= b'(0).
\end{align}
Moreover, as 
\begin{align*}
    \lim_{t\to 0^+}\big( a(t)+a'(t)t\big)= \lim_{t\to 0^+}b'(t) = b'(0)=\lim_{t\to 0^+}a(t),
\end{align*}
one has that \begin{align}
    \label{dec3}
    \lim_{t\to 0^+}a'(t)t=0.
\end{align}
Accordingly, with \eqref{dec1} in force, we also extend at $t=0$ the function $\theta_a$, defined in \eqref{theta},  and set:
\begin{align}
    \label{theta0}
    \theta_a(0)=0.
\end{align}

Given $p>1$ and $\varepsilon \in (0,1)$, define the function $a_{p,\varepsilon} : [0, \infty) \to [0, \infty)$ as
\begin{equation}\label{aeps}
a_{p,\varepsilon} (t) = \frac{(\varepsilon + t^2)^{\frac{p-2}{2}} +
\varepsilon}{1 + \varepsilon (\varepsilon + t^2)^{\frac{p-2}{2}} } \quad
\quad \hbox{for $t \geq 0$.}
\end{equation}
Notice that, for $\varepsilon =0$, the formula \eqref{aeps} reproduces the function 
$a_p(t)=t^{p-2}$ for $t>0$. 
\\ One can verify that $a_{p,\varepsilon} $ shares its monotonicity property with $t^{p-2}$, and:
\begin{equation}\label{cinf}
a_{p,\varepsilon} \in C^1 ([0 , \infty ));
\end{equation}
\begin{align}\label{abound}
\varepsilon \leq    a_{p,\varepsilon} (t)   \leq \varepsilon ^{-1} \quad
\hbox{for $t \geq 0$.}
\end{align}
Also, given $t_0>0$, there exist  constants $c_1(p,t_0)$ and $c_2(p,t_0)$ such that
    \begin{align}
        \label{abound1}
      c_1(p,t_0)  \leq a_{p,\varepsilon}(t_0) \leq c_2(p,t_0) \quad \text{for $\varepsilon \in (0,1)$.}
    \end{align}
Set
\begin{align}\label{mpMp}
     m_p=\min\{p,2\}  \quad \text{and} \quad M_p = \max\{p,2\}. 
\end{align}
Then,
\begin{align}\label{indici}
m_p-2 \leq i_{a_{p,\varepsilon}} \leq s_{a_{p,\varepsilon}} \leq
M_p-2.
\end{align}
Moreover, let $b_{p,\varepsilon}$ and $B_{p,\varepsilon}$ be defined as in
\eqref{b} and \eqref{B}, respectively, with $a$ replaced with
$a_{p,\varepsilon}$.
 Owing to \eqref{jan200} with $a$ replaced with $a_{p,\varepsilon}$, 
\begin{align}
    \label{mar42}
    m_p-1 \leq i_{b_{p,\varepsilon}}\leq s_{b_{p,\varepsilon}} \leq M_p-1 \qquad \text{for $\varepsilon \in (0,1)$.}
\end{align}
Hence, 
\begin{align}
    \label{feb50}
    \text{the function $b_{p,\varepsilon}$ is increasing,}
\end{align}
and
\begin{align}
    \label{nov48}
(m_p-1)b_{p,\varepsilon} (t)  \leq b_{p,\varepsilon}'(t) t \leq (M_p-1)b_{p,\varepsilon} (t) \quad \text{for $t \geq 0$.}
\end{align}
Moreover, from \eqref{nov48}
 and \eqref{abound1} one infers that 
\begin{align}
    \label{feb60}
    c_1(p) \min\{t^{m_p-1}, t^{M_p-1}\} \leq b_{p,\varepsilon}(t) \leq c_2(p) \max\{t^{m_p-1}, t^{M_p-1}\} \quad \text{for $t\geq 0$.}
\end{align}
Thanks to \eqref{jan201}, with $a$ replaced with $a_{p,\varepsilon}$, and \eqref{indici}, 
    \begin{align}
        \label{feb33}
        m_p \leq i_{B_{p,\varepsilon}}\leq s_{B_{p,\varepsilon}}\leq M_p.
    \end{align}
Therefore,
\begin{align}
    \label{mar43}
    m_p\,B_{p,\varepsilon} (t)  \leq b_{p,\varepsilon}(t) t \leq M_p\,B_{p,\varepsilon} (t) \quad \text{for $t \geq 0$.}
\end{align}
 Notice that
\begin{align}\label{2026-1}
    \tfrac 14 a_{p,\varepsilon}(1/2)   \leq \int_{\frac 12}^1 a_{p,\varepsilon}(t)t dt \leq B_{p,\varepsilon} (1) \leq \tfrac 12 a_{p,\varepsilon}(1) \quad \text{for $t\geq 0$.}
\end{align}
Via \eqref{abound1}, \eqref{mar43} and \eqref{2026-1} one can deduce that
\begin{align}
    \label{feb35}
    c_1(p) \min\{t^{m_p}, t^{M_p}\} \leq B_{p,\varepsilon}(t) \leq c_2(p) \max\{t^{m_p}, t^{M_p}\} \quad \text{for $t\geq 0$,}
\end{align}
for suitable constants $c_1(p)$ and $c_2(p)$.
Owing to \eqref{feb35}, for every $\varepsilon \in (0,1)$,
\begin{align}
    \label{mar45}
    B_{p,\varepsilon}(t) \leq \begin{cases}
        c(t^2 + 1) &\quad \text{if $p < 2$}
        \\ c(t^p + 1) & \quad \text{if $p\geq 2$,}
    \end{cases}
\end{align}
for some constant $c=c(p)$ and
for $t\geq 0$, and 
\begin{align}
    \label{nov41}
    B_{p,\varepsilon}(t) \geq \begin{cases}
        c(t^p - 1) &\quad \text{if $p < 2$}
        \\ c(t^2 - 1) & \quad \text{if $p\geq 2$,}
    \end{cases}
\end{align}
for some constant $c=c(p)$ and
for $t\geq 0$. Moreover,
\begin{align}
    \label{mar45'}
    \widetilde {B_{p,\varepsilon}}(t) \leq \begin{cases}
        c(t^{p'} + 1) &\quad \text{if $p < 2$}
        \\ c(t^2 + 1) & \quad \text{if $p\geq 2$,}
    \end{cases}
\end{align}
for some constant $c=c(p)$ and
for $t\geq 0$. Given any $R>0$,
\begin{equation}\label{convb}
 \lim _{\varepsilon \to 0} b_{p,\varepsilon} = b_p \qquad
\hbox{uniformly in $[0, R]$,}
\end{equation}
and  
\begin{equation}\label{convB}
 \lim _{\varepsilon \to 0} B_{p,\varepsilon} = B_p \qquad \hbox{uniformly in $[0, R]$.}
\end{equation}
Furthermore,
define the function 
$\mathcal A_{p,\varepsilon} : \mathbb R^{n\times n} \to \mathbb R^{n\times n} $
as 
\begin{align}\label{Aeps}
\mathcal A_{p,\varepsilon} (\xi) = a_{p,\varepsilon} (|\xi|)\xi \qquad \text{for $\xi \in \mathbb R^{n\times n}$.}
\end{align}
Then,
\begin{align}\label{conva}
 \lim _{\varepsilon \to 0} \mathcal A_{p,\varepsilon} = \mathcal A_p \qquad
\hbox{uniformly in any ball $B_R$.}
\end{align}
One has that
\begin{align}
    \label{nov60'}
    |\mathcal A_{p,\varepsilon}(\xi) -  \mathcal A_{p,\varepsilon} (\eta)| \eqsim a_{p,\varepsilon}(|\xi|+|\eta|) |\xi - \eta| \eqsim b_{p,\varepsilon}'(|\xi|+|\xi-\eta|) |\xi - \eta|  \qquad \text{for $\xi, \eta \in \mathbb R^{n\times n}$,}
\end{align}
and for every $\varepsilon \in (0, 1)$,
with equivalence constants depending only on $p$. Equation \eqref{nov60'} follows via an application of \cite[Lemma 40 and Lemma 41]{DieningFornasierWank20}.

\begin{remark}
  \label{rem:Orlicz1}
{\rm As mentioned in Section \ref{S:intro}, systems more general than \eqref{eq:system}, where $\mathcal A_p(\eps u)$ is replaced with a law for the stress tensor $\mathcal{A}(\eps u)$ given by \eqref{mathcalA}, can be investigated via the methods of this paper. The Carreau law in~\eqref{Apnu} is one instance. 
This would require suitable assumptions on the indices $i_a$ and $s_a$ of the function~$a$ which defines $\mathcal A$, and  approximations analogous to those introduced for $a_p$.
In particular,
the  functions $a_\varepsilon : [0,\infty) \to (0,\infty)$, defined, for $\varepsilon \in (0,1)$, as 
$$a_\varepsilon (t)=  \frac{a\big(\sqrt{\varepsilon + t^2}\big) +
\varepsilon}{1 + \varepsilon a\big(\sqrt{\varepsilon + t^2}\big) } \quad
\quad \hbox{for $t \geq 0$,}$$
play  the role of $a_{p,\varepsilon}$.
The function $a_\varepsilon$, and the functions $b_\varepsilon$, $B_\varepsilon$, and $\mathcal A_\varepsilon$ defined accordingly, enjoy  properties parallel to those of $a_{p,\varepsilon}$, $b_{p,\varepsilon}$, $B_{p,\varepsilon}$, and $\mathcal A_{p,\varepsilon}$. Most of these properties are a consequence of the inequalities
\begin{align*}
    \min\{i_a, 0\}\leq i_{a_\varepsilon}\leq s_{a_\varepsilon}\leq \max\{s_a, 0\},
\end{align*}
of which \eqref{indici} is a special case.}
\end{remark}

\section{Pointwise differential inequalities}
\label{sec:pointwise}

In this section we derive new critical pointwise identities and inequalities involving the symmetric gradient $p$-Laplace operator. Since the inequalities have to be applied to operators which regularize the symmetric $p$-Laplacian, they are proved for  symmetric gradient  operators with  more general Orlicz growth. 
The presence of the symmetric gradient makes them substantially more involved than those offered for the standard gradient in  \cite{CianchiMazya18,CianchiMazya19,BalciDieningCianchiMazya22}. For instance, whereas the latter require linear algebra techniques for matrices, the inequalities for the symmetric $p$-Laplacian that will presented entail dealing with tensors.

Given a function $u: \Omega \to \mathbb R^{m}$, we denote  by $\nabla u : \Omega \to \mathbb R^{m\times n}$ its gradient, namely a matrix-valued function whose $m$ rows are the gradients $\nabla u_i \in \mathbb R^{1\times n}$ of the components $u_i$, $i=1, \dots m$, of $u$. The notation $\nabla^T u$ will be employed to denote  $(\nabla u)^T$. The divergence  of a matrix-valued function $M: \Omega \to \mathbb R^{m\times n}$ is a function $\divergence M : \Omega \to \mathbb R^{m\times 1}$ whose components are the divergences of the rows of $M$. 

We begin by collecting some identities involving the symmetric gradient.
Assume that $u : \Omega \to \RRn$ is such that
$u\in W^{2,1}_{\rm loc}(\Omega)$.
Define the \emph{symmetric Laplace operator} as 
$$\Delta^{\sym} u= \divergence \eps u.$$
Namely, $\Delta^{\sym} u= \divergence(\mathcal A_2(\eps u))$.
Computations show that
\begin{equation} \label{eq:5}
	\Delta^{\sym} u = \tfrac{1}{2}\Delta u + \tfrac{1}{2}\nabla^T \divergence u. 
\end{equation}
More generally, if $p>1$, then
\begin{align}
	\divergence(\mathcal A_p(\eps u))=\tfrac 12|\eps u |^{p-2}\left(\Delta u + \nabla ^T \divergence u\right) + 
     (p-2)|\eps u|^{p-2}\frac{\eps u}{|\eps u|} (\nabla |\eps u|)^T.
\end{align}

The full second order gradient $\nabla ^2u$ and $\nabla \eps u$ are comparable, in the sense that there exists an invertible linear map $L:\RRnn\rightarrow\RRnn$ such that 
\begin{align}
    \label{nov1}
    \nabla ^2 u =L \nabla \eps u.
\end{align}
 This is a consequence of the identity
\begin{equation}\label{feb15}
	\partial_i \partial_j u_k = \partial_i \eps_{jk} u+\partial_j \eps_{ik} u-\partial_k \eps_{ij} u.
\end{equation}

\begin{lemma}
    \label{lemma:equiv}
   Assume that the function $a: (0, \infty) \to (0, \infty)$ is such that $a\in C^1((0, \infty))$ and its indices satisfy the conditions \eqref{infsup}. Let $\mathcal A$ be the function defined as in \eqref{mathcalA}.
   \\
(i) Let $u\in C^2(\Omega)$. 
 Then 
\begin{align}
    \label{nov50}
    \abs{\nabla (A(\eps u))}\eqsim a(|\eps u|)|\nabla \eps u|\eqsim a(|\eps u|)|\nabla^2 u|
\end{align}
in the set $\{x\in \Omega: \eps u \neq 0\}$,
up to equivalence constants depending on $i_a$ and $s_a$. Under the stronger assumption 
\eqref{dec1}, Equation \eqref{nov50} holds for every $x\in \Omega$, with the function $a$ extended at $0$ as in \eqref{a(0)}.

\noindent(ii) Let $u\in W^{2,2}_{\rm loc}(\Omega)$. Assume, that the condition \eqref{dec1} is in force and
\begin{align*}%\label{quadratic}
    \lambda \leq a(t) \leq \lambda^{-1} \quad \text{for $t>0$,}
\end{align*}
for some constant $\lambda \in (0, 1)$.
 Then Equation \eqref{nov50} holds   for a.e. $x\in \Omega$.
\end{lemma}
\begin{proof} (i) Let $w=(w_1, \dots , w_N)\in C^2(\Omega, \RRN)$, with $N \in \setN$. Then,
\begin{align}
    \label{nov53}
    \abs{\nabla (\mathcal A(w))} &= \sqrt{\sum_{i=1}^N
\bigg|a(|w|)\nabla w_i + a'(|w|) \nabla w_i\frac{w}{|w|} \otimes w \bigg|^2}
    \\ \nonumber & =    a(|w|)\sqrt{\sum_{i=1}^N
\bigg|\nabla w_i\bigg({\rm I}_N + \frac{a'(|w|)|w|}{a(|w|)} \frac{w}{|w|} \otimes \frac{w}{|w|}\bigg)}\bigg|^2
\end{align}
where ${\rm I}_N$ denotes the  identity matrix of dimension $N$. The identity \eqref{nov53} plainly holds at every point in $\Omega$ where $w\neq 0$. Under the assumption \eqref{dec1}, it continues to hold also where $w=0$, thanks to \eqref{dec3}, provided that  
 we set
\begin{align*}
    a'(|w|) \frac{w}{|w|} \otimes w =0 \qquad \text{if $w=0$.}
\end{align*}
 Now, given any unit vector $\upsilon \in \mathbb R^N$, consider the symmetric matrix $M$ defined as 
$$M= \mathrm{I}_N + \frac{a'(|w|)|w|}{a(|w|)} \upsilon \otimes \upsilon,$$
and denote by $\lambda_M$ and $\Lambda_M$ its smallest and largest eigenvalue, respectively. We claim that
\begin{align}
    \label{nov65}
0< 1+i_a \leq  \lambda_M \leq \Lambda_M \leq 1+s_a<\infty.
\end{align}
Indeed,  owing to \eqref{infsup},
\begin{align}\label{nov68}
    -1<i_a\leq \frac{a'(|w|)|w|}{a(|w|)}\leq s_a <\infty.
\end{align} 
One has that
\begin{align*}
   \upsilon M = \upsilon\big(1+ a'(|w|)|w|/a(|w|)\big).
\end{align*}
On the other hand, if $\upsilon \cdot w=0$, then
\begin{align*}
   \upsilon M  =0.
\end{align*}
Therefore, $\big(1+ a'(|w|)|w|/a(|w|)\big)$ and $0$ are eigenvalues of $M$ of multiplicity $1$ and $n-1$, respectively. Equation \eqref{nov65} hence follows via \eqref{nov68}. 
As the eigenvalues of $M^2$ agree with the squares of those of $M$, one has that
\begin{align*}
 (1+i_a)^2 \leq  \lambda_{M^2} \leq \Lambda_{M^2} \leq (1+s_a)^2.
\end{align*}
On the other hand,
\begin{align*}
    |\nabla w_i M|= \sqrt{\nabla w_i M \cdot \nabla w_i M}= \sqrt{ \nabla w_i\cdot \nabla w_i M^2}
\end{align*}
for $i=1, \dots N$.
Hence,
\begin{align}
    \label{nov71}
  (1+i_a)|\nabla w_i|\leq  |\nabla w_i M|\leq (1+s_a)|\nabla w_i|.
\end{align}
Coupling \eqref{nov53} with \eqref{nov71} tells us that
\begin{align}
    \label{nov72}
 (1+i_a) a(|w|) |\nabla w|   \leq \abs{\nabla (\mathcal A(w))}\leq (1+s_a) a(|w|) |\nabla w|.
\end{align}
The first equivalence in \eqref{nov50} follows via and application of \eqref{nov72} with $N={n^2}$ and $w=\eps u$. The second equivalence is a consequence of \eqref{feb15}.
\\
(ii) The proof follows along the same lines as in case (i). One just has to observe that, under the current assumptions on the function $a$, Equation \eqref{nov53} holds for a.e. $x\in \Omega$ thanks to the chain rule for vector valued Sobolev functions \cite{Marcus-Mizel}.
\end{proof}

    \begin{lemma}
        \label{lem:DeltaSymDelta}
        Assume that $u\in C^3(\Omega)$. Then, 
        \begin{align}
            \label{DeltaSymDelta}
            \Delta^\sym u \cdot \Delta u =  \abs{\nabla \eps u}^2 + \divergence \Big((\Delta u)^T\eps u  -  \nabla (\tfrac 12\abs{\eps u}^2) \Big).
        \end{align}
    \end{lemma}

    \begin{proof}
        The identity \eqref{DeltaSymDelta} follows from the following chain:
        \begin{align*}
            \Delta^\sym u \cdot \Delta u &= \sum_{ijl} \partial_i \eps_{ij}u \,\partial_{ll} u_j
            =\sum_{ijl} \partial_i (\eps_{ij}u\, \partial_{ll} u_j) - \eps_{ij} u \, \partial_{lli} u_j
            \\
            &=\sum_{ijl} \partial_i (\eps_{ij}u\, \partial_{ll} u_j) - \eps_{ij} u \, \partial_{ll} \eps_{ij}u
            \\
            &=\sum_{ijl} \partial_i (\eps_{ij}u\, \partial_{ll} u_j) -\partial_l ( \eps_{ij} u \, \partial_l \eps_{ij}u) + \partial_l \eps_{ij}u \, \partial_l \eps_{ij}u 
            \\
            &=  \divergence ((\Delta u)^T\eps u  )-\divergence (\tfrac 12 \nabla  (\abs{\eps u}^2) ) +\abs{\nabla \eps u}^2.
        \end{align*}
Notice that the third equality holds 
 by a symmetrization argument between the indices $i$ and $j$. 
    \end{proof}

   The pointwise identity contained in the following lemma holds for any dimension $n\geq 2$, but  will play a role in the proof of our main result only for $n\leq 7$.

    \begin{lemma}[First pointwise identity]
        \label{lem:pw_identity_lowdim}
        Let $a\colon (0,\infty ) \to (0,\infty)$ be such that $a\in C^1((0,\infty))$.  Let $\theta_a$ and  $\mathcal A$ be the functions defined as in \eqref{theta} and \eqref{mathcalA}.
 Assume that $u\in C^3(\Omega)$. Then,  
        \begin{align}
                \label{pw_identity_lowdim}
                2a(\abs{\eps u})\,\divergence (\mathcal A(\eps u)) \cdot \Delta u  &= a^2(\abs{\eps u})\Big[ \abs{\nabla \eps u}^2 +\tfrac 12 \abs{\Delta u}^2+ 2 \theta_a(\abs{\eps u}) \abs{\nabla\abs{\eps u}}^2  +\tfrac 12 \nabla \divergence u \Delta u \Big]         \\ \nonumber 
                &\quad  +\divergence \Big(a^2(\abs{\eps u}) \big((\Delta u)^T\eps u -\tfrac 12 \nabla (\abs{\eps u}^2) \big) \Big)
        \end{align}
         in the set $\set{x\in \Omega: \eps u(x) \neq 0}$. Under the additional assumption \eqref{dec1}, the identity \eqref{pw_identity_lowdim} holds for every $x\in \Omega$, provided that, according to \eqref{theta0}, the term depending on  $\theta_a$ is taken to be $0$ at every point where $\eps u=0$.
    \end{lemma}
    
    \begin{proof}[Proof of Lemma \ref{lem:pw_identity_lowdim}] For simplicity of notation, we set 
         $\theta = \theta_a (\abs{\eps u})$ and $a= a(\abs{\eps u})$ throughout this proof.
\\ First,
        consider any point $x\in \Omega$ such that $\eps u (x)\neq 0$. Then, the following computations can be verified via standard calculus rules.
        Since $\nabla a =  a \theta \abs{\eps u}^{-1} \nabla \abs{\eps u}$, one has that
        \begin{align}
                \label{snow1}
                a \divergence (a\,\eps u) \cdot \Delta u &= a^2 \Delta^\sym u \cdot \Delta u + a^2\, \theta \nabla \abs{\eps u} \frac{\eps u}{\abs{\eps u}}\Delta u.  
        \end{align}
       Coupling Equation \eqref{snow1} with \eqref{DeltaSymDelta} yields:
        \begin{align*}
            a \divergence (a\,\eps u) \cdot \Delta u &= a^2 \Big( \abs{\nabla \eps u}^2 + \divergence \big((\Delta u)^T\eps u - \nabla (\tfrac 12 \abs{\eps u}^2) \big)+\theta \nabla \abs{\eps u} \frac{\eps u}{\abs{\eps u}}\Delta u\Big) 
            \\
            &= a^2 \Big( \abs{\nabla \eps u}^2 +\theta \nabla \abs{\eps u} \frac{\eps u}{\abs{\eps u}}\Delta u\Big) + \divergence \big( a^2((\Delta u)^T\eps u  - \nabla (\tfrac 12\abs{\eps u}^2)) \big)
            \\
            &\quad - \nabla (a^2)  \big(\eps u \Delta u - \nabla^T (\tfrac 12\abs{\eps u}^2)\big).
        \end{align*}
        On the other hand,
        \begin{align*}
            \nabla (a^2) \big(\eps u \Delta u - \nabla^T (\tfrac 12\abs{\eps u}^2)\big)&=2a^2 \frac{\theta}{\abs{\eps u}} \nabla \abs{\eps u} \big(\eps u \Delta u - \abs{\eps u}\nabla^T \abs{\eps u}\big)
            \\
            &=2a^2 \theta \nabla \abs{\eps u} \Big(\frac{\eps u}{\abs{\eps u}} \Delta u -\nabla^T \abs{\eps u}\Big).
        \end{align*}
        Combining the last two identities tells us that
        \begin{align}
            \label{snow2}
            \begin{aligned}
                a \, \divergence (a\, \eps u) \cdot \Delta u &= a^2 \Big( \abs{\nabla \eps u}^2 +2\theta \abs{\nabla \abs{\eps u}}^2 - \theta \nabla \abs{\eps u}\frac{\eps u}{\abs{\eps u}}\Delta u \Big)
                \\
                &\quad + \divergence \big( a^2((\Delta u)^T\eps u  - \nabla (\tfrac 12\abs{\eps u}^2)) \big).
            \end{aligned}
        \end{align}
        The  identity 
    \eqref{pw_identity_lowdim} follows by adding \eqref{snow1} and \eqref{snow2} and recalling \eqref{eq:5}.
\\
Next, assume that \eqref{dec1} holds and let  $x\in \Omega$ be such that $\eps u(x)=0$.
        One has that
        \begin{align}
            \label{2026-5}
            \lim_{\xi \to 0}\nabla_\xi (a(|\xi|)\xi)= \lim_{\xi \to 0} \left[a(|\xi|){\rm I}_n+ a'(|\xi|)\frac{\xi}{|\xi|}\otimes \xi\right]= a(0){\rm I}_n,
        \end{align}
        where the second equality holds owing to \eqref{dec3}. Thus, the function $\xi \mapsto a(|\xi|)\xi$ is continuously differentiable also at $0$, and from the chain rule we infer that
        \begin{align}
            \label{snow3}
            a\divergence(a \eps u)\cdot \Delta u &=  a^2 \Delta ^\sym u \cdot \Delta u.
        \end{align}
       Coupling the latter equality with \eqref{DeltaSymDelta} yields
        \begin{align}
            \label{snow4}
            a\divergence(a \eps u) \cdot \Delta u = a^2 \Big(\abs{\nabla \eps u}^2 + \divergence \big(  (\Delta u)^T\eps u- \nabla  (\tfrac 12 \abs{\eps u}^2) \big) \Big).
        \end{align} 
        Now, set
        \begin{align*}
            F(z)= a^2 (\abs{\eps u(z)})\Big((\Delta u)^T\eps u(z) - \nabla \big(\tfrac 12 \abs{\eps u(z)}^2\big) \Big)    .
        \end{align*}
The function $F$ is continuous at $x$. Moreover, thanks to the property \eqref{dec3}, 
  \begin{align*}
            \lim_{z\to x} \nabla F(z) = a^2(0) \nabla \Big((\Delta u)^T\eps u - \nabla \big(\tfrac 12 \abs{\eps u}^2\big)\Big)(x).
        \end{align*}
        This shows that $F$ is also continuously differentiable at $x$ and
        $$\nabla F(x) = a^2(0) \nabla \Big((\Delta u)^T\eps u - \nabla \big(\tfrac 12 \abs{\eps u}^2\big)\Big)(x).$$
        Hence,
        \begin{align}
            \label{snow5}
            \divergence \Big(a^2 \big((\Delta u)^T\eps u - \nabla(\tfrac 12 \abs{\eps u}^2)\big) \Big)=a^2 \divergence \Big((\Delta u)^T\eps u  - \nabla(\tfrac 12 \abs{\eps u}^2) \Big).
        \end{align} 
        Adding the identities \eqref{snow3} and \eqref{snow4} and combining the resultant  with \eqref{snow5} establish  the identity \eqref{pw_identity_lowdim} also at every point $x$ where $\eps u (x)=0$. The   proof is complete.
    \end{proof}

The next lemma provides us with an alternative identity,
to be used in the proof of our main result for  combinations of the values of $n$ and $p$ complementary to those covered via Lemma \ref{lem:pw_identity_lowdim}.

    \begin{lemma}
        [Second pointwise identity]
        \label{lem:pw_identity_alldim}
         Let $a\colon (0,\infty ) \to (0,\infty)$ be such that $a\in C^1((0,\infty))$.
         Let $\theta_a$ and  $\mathcal A$ be the functions defined as in \eqref{theta} and \eqref{mathcalA}.
         Assume that $u\in C^3(\Omega)$. Then,  
        \begin{align}
            \label{pw_identity_alldim}
        \begin{aligned}
                \lefteqn{2a(\abs{\eps u})\,\divergence (\mathcal A(\eps u)) \cdot \Delta u} \quad &
                \\ 
                &=a^2(\abs{\eps u})\bigg[ \abs{\nabla \eps u}^2\!+ \tfrac 12 \abs{\Delta u}^2 \!+\tfrac 12 \abs{\nabla \divergence u}^2 
               \!+2\theta_a (\abs{\eps u})\abs{\nabla \abs{\eps u}}^2\!
               \\
               &\hspace{20mm} - \theta_a (\abs{\eps u}) \nabla \abs{\eps u}(\Delta u \!-\!\nabla^T \!\divergence u)\frac{\divergence u}{\abs{\eps u}} \bigg]
                \\ 
                &\quad +\divergence \Big(a^2(\abs{\eps u}) \big( (\Delta u)^T \eps u  -\tfrac 12 \nabla (\abs{\eps u}^2) +\tfrac 12 \divergence u ((\Delta u)^T - \nabla \divergence u)\big) \Big)
            \end{aligned}
        \end{align}
 in the set $\set{x\in \Omega: \eps u(x) \neq 0}$. Under the additional assumption \eqref{dec1}, the identity \eqref{pw_identity_alldim} holds for every $x\in \Omega$, provided that, according to \eqref{theta0}, the term depending on  $\theta_a$ is taken to be $0$ at every point where $\eps u=0$.
    \end{lemma}

    \begin{proof} As in the previous proof, 
        we adopt the simplified notation $a= a(\abs{\eps u})$ and $\theta = \theta_a (\abs{\eps u})$.
\\ Fix $x\in \Omega$. Assume first  that $\eps u(x)\neq 0$. Then,
        \begin{align*}
                \nabla \divergence u \cdot \Delta u & = \sum_{kl} \partial_k \divergence u \,\partial_{ll}u_k
               =\sum_{kl} \partial_k (\divergence u \, \partial_{ll}u_k) - \divergence u \,\partial_{kll}u_k
                \\ \nonumber 
                &= \sum_{kl} \partial_k (\divergence u \, \partial_{ll}u_k) - \partial_l (\divergence u \, \partial_{lk}u_k) + \partial_l \divergence u \, \partial_{lk} u_k
                \\
                &= \divergence (\divergence u \, (\Delta u)^T - \divergence u \, \nabla \divergence u) + \abs{\nabla \divergence u}^2.
        \end{align*}
       Hence,
        \begin{align*}
            a^2\,\nabla \divergence u  \Delta u &=a^2\Big(\divergence (\divergence u \, (\Delta u )^T - \divergence u \, \nabla \divergence u) + \abs{\nabla \divergence u}^2\Big)
            \\
            &=\divergence \Big( a^2 \divergence u \big((\Delta u)^T -  \nabla \divergence u\big)\Big) + a^2\, \abs{\nabla \divergence u}^2
            - \nabla (a^2) \divergence u \big(\Delta u - \nabla^T \divergence u\big)
            \\
            &=\divergence \Big( a^2 \divergence u \big((\Delta u)^T -  \nabla  \divergence u\big)\Big) + a^2\, \abs{\nabla \divergence u}^2
            - 2a^2 \theta\, \frac{\divergence u}{\abs{\eps u}} \nabla \abs{\eps u}\big(\Delta u - \nabla^T \divergence u\big).
        \end{align*}
        Coupling  the latter identity with \eqref{pw_identity_lowdim} yields \eqref{pw_identity_alldim}.
\\ Now, suppose that $\eps u(x)=0$. One can argue as above, and make use of the fact that the last term in the second line of the latter chain is missing, since
        \begin{align*}
            a^2\divergence \Big( \divergence u \big((\Delta u)^T -  \nabla \divergence u\big)\Big) = \divergence \Big( a^2 \divergence u \big((\Delta u)^T -  \nabla \divergence u\big)\Big).
        \end{align*}
This identity can be verified analogously to  \eqref{snow5}.
    \end{proof}

    In Lemmas \ref{lem:pw_ineq_lowdim} and \ref{lem:pw_ineq_alldim} pointwise estimates for the expression $a(|\eps u|)^2\abs{\nabla \eps u}^2 \eqsim\abs{\nabla (A(\eps u))}$ are derived from Lemmas \ref{lem:pw_identity_lowdim} and \ref{lem:pw_identity_alldim}, respectively, under suitable assumptions on the indices $i_a$ and $s_a$. Note that  Lemma \ref{lem:pw_ineq_lowdim} only holds for  $n\leq 7$. 

    \begin{lemma}[First pointwise differential inequality]
    \label{lem:pw_ineq_lowdim}
        Let $2\leq n \leq 7$ and let $a:(0,\infty) \to (0,\infty)$ 
        be such that the condition  \eqref{dec1} holds and
        \begin{align}
        \label{p_bounds_lowdim}
           - \frac{8-n}{16} <i_a \leq  s_a<\infty.   
        \end{align}
                 Let   $\mathcal A$ be the function defined by \eqref{mathcalA}.
        If $u\in C^3(\Omega)$, then  
        \begin{align}
            \label{pwineq1}
            \begin{aligned}
                \abs{\divergence (\mathcal A(\eps u))}^2&\geq c_1\, a^2(\abs{\eps u}) \abs{\nabla \eps u}^2 + c_2\,\divergence\Big(a^2(\abs{\eps u}) \big((\Delta u)^T\eps u -\tfrac 12 \nabla (\abs{\eps u}^2) \big) \Big) \quad \text{in $\Omega$}
            \end{aligned}
        \end{align}
        for some positive constants  $c_1$ and $c_2$  depending only on $i_a$, $s_a$, and $n$.
    \end{lemma}

    \begin{proof}
    Let $\theta$ and $a$ be defined as in the previous proofs.
  Let $\delta\in (0,\frac 12)$ be a parameter to be chosen later. By Young's inequality, 
        \begin{align*}
            2a\, \divergence (a\,\eps u )\cdot \Delta u&\leq \delta a^2\,\abs{\Delta u}^2 + \delta^{-1}\abs{\divergence (a\,\eps u )}^2. 
        \end{align*}
        Owing to Lemma \ref{lem:pw_identity_lowdim}, the inequality \eqref{pwineq1} will follow if we show that
        \begin{align}
            \label{snow6}
            \abs{\nabla \eps u}^2 + \big(\tfrac 12 - \delta \big) \abs{\Delta u}^2 + 2\theta \abs{\nabla \abs{\eps u}}^2 + \tfrac 12 \nabla \divergence u \Delta u &\geq c\, \abs{\nabla \eps u}^2
        \end{align}
        for some constants $c>0$ and $\delta >0$  depending only on $n$, $i_a$, and $s_a$. As agreed in the statement of Lemma \ref{lem:pw_identity_lowdim}, in the set $\{x\in \Omega: \eps u(x)=0\}$ the term involving $\theta$ in the inequality  \eqref{snow6} is just missing. From Young's inequality again and the inequality
       $$\abs{\nabla \divergence u}^2\leq n \abs{\nabla \eps u}^2$$ we deduce that
        \begin{align*}
            \frac 12 \abs{\nabla \divergence u \Delta u} &\leq \frac{1-2\delta}{2} \abs{\Delta u}^2 + \frac{1}{8(1-2\delta)}\abs{\nabla \divergence u}^2
           \leq \frac{1-2\delta}{2} \abs{\Delta u}^2 + \frac{n}{8(1-2\delta)}\abs{\nabla \eps u}^2.
        \end{align*}
        Consequently,
        \begin{align}
            \label{snow7}
            \abs{\nabla \eps u}^2 + \big(\tfrac 12 - \delta \big) \abs{\Delta u}^2 + 2\theta \abs{\nabla \abs{\eps u}}^2 + \tfrac 12 \nabla \divergence u \Delta u &\geq \Big(1-\frac{n}{8(1-2\delta)}\Big)\abs{\nabla \eps u}^2 + 2\theta \abs{\nabla \abs{\eps u}}^2.
        \end{align}
        If $\theta \geq 0$, then, since we are assuming that $n\leq 7$, the inequality \eqref{snow6} follows from \eqref{snow7} for sufficiently small $\delta$. 
        \\ If $\theta <0$, then  $2\theta \abs{\nabla \abs{\eps u}}^2 \geq 2\theta \abs{\nabla\eps u}^2$. The latter inequality, in combination  with \eqref{snow7},  implies that
        \begin{align}\label{2026-7}
            \abs{\nabla \eps u}^2 + \big(\tfrac 12 - \delta \big) \abs{\Delta u}^2 + 2\theta \abs{\nabla \abs{\eps u}}^2 + \tfrac 12 \nabla \divergence u \Delta u &\geq \Big(1-\frac{n}{8(1-2\delta)}+2\theta\Big)\abs{\nabla \eps u}^2.
        \end{align}
Now, observe that, since $\theta\geq i_a$ and the assumption \eqref{p_bounds_lowdim} is in force,
we have that
$1-\frac{n}{8}+2\theta>0$.
Hence, the inequality \eqref{snow6} follows from \eqref{2026-7}, provided that $\delta$ is chosen small enough.
    \end{proof}

    \begin{lemma}[Second pointwise differential inequality]
    \label{lem:pw_ineq_alldim}
         Let $n\in \setN$. Let $a\colon (0,\infty ) \to (0,\infty)$ be such that  the condition \eqref{dec1} is fulfilled and
        \begin{align}
        \label{p_bounds_alldim}
-\frac{1}{\sqrt{n+1}+1}< i_a \leq s_a< \frac{1}{\sqrt{n+1}-1}.  
        \end{align}
        Let   $\mathcal A$ be the function defined by   \eqref{mathcalA}. 
        If $u\in C^3(\Omega)$,  then
        \begin{multline}
            \label{pwineq2}
                \abs{\divergence (\mathcal A(\eps u))}^2\geq c_1\, a^2(\abs{\eps u}) \abs{\nabla \eps u}^2 
                \\  +c_2\,\divergence \Big(a^2(\abs{\eps u}) \big( (\Delta u)^T\eps u -\tfrac 12 \nabla (\abs{\eps u}^2) +\tfrac 12 \divergence u ((\Delta u)^T - \nabla \divergence u)\big) \Big) \quad \text{in $\Omega$}
        \end{multline}
        for some positive constants $c_1$ and $c_2$ depending only on $i_a$, $s_a$, and $n$.
    \end{lemma}

    \begin{proof}
        We define $\theta$ and $a$ as in the preceding proofs and use Young's inequality to deduce that
        \begin{align*}
            2a\, \divergence (a\,\eps u )\cdot \Delta u&\leq \delta a^2\,\abs{\Delta u}^2 + \delta^{-1}\abs{\divergence (a\,\eps u )}^2, 
        \end{align*}
        where $\delta\in (0,\frac 12)$ is a parameter to be chosen later. Thanks to Lemma \ref{lem:pw_identity_alldim}, the proof of the inequality \eqref{pwineq2}  is reduced to showing 
        \begin{align}
            \label{snow8}
                \abs{\nabla \eps u}^2 \!+\! (\tfrac 12 - \delta) \abs{\Delta u}^2\!+\! \tfrac 12 \abs{\nabla \divergence u}^2 \!+\! 2\theta \abs{\nabla \abs{\eps u}}^2 %&
                %\\ \nonumber
                \!+\!\theta \nabla \abs{\eps u}(\nabla^T\divergence u \!-\!\Delta u)\frac{\divergence u}{\abs{\eps u}}&\geq c \abs{\nabla \eps u}^2
        \end{align}
        in~$\Omega$, for suitable constants $c>0$ and $\delta\in (0,\frac 12)$, both depending only on $n$, $i_a$, and $s_a$. As specified in the statement of 
        Lemma \ref{lem:pw_identity_alldim},
         the terms involving $\theta$ in the inequality \eqref{snow8} are missing at those points  $x\in\Omega$ where $\eps u(x)=0$. 
         \\
         It suffices to establish \eqref{snow8} in the case when $\delta=0$, namely
         \begin{align}
            \label{snow80}
                \abs{\nabla \eps u}^2 + \tfrac 12  \abs{\Delta u}^2+ \tfrac 12 \abs{\nabla \divergence u}^2 + 2\theta \abs{\nabla \abs{\eps u}}^2
                +\theta \nabla \abs{\eps u}(\nabla^T\divergence u -\Delta u)\frac{\divergence u}{\abs{\eps u}}&\geq c \abs{\nabla \eps u}^2.
        \end{align}
    Indeed, if \eqref{snow80} holds for $\delta =0$ and for some constant $c>0$, then we can use the inequality $\abs{\Delta u}^2 \leq c(n) \abs{\nabla \eps u}^2$ to deduce that \eqref{snow8}  holds for sufficiently small $\delta\in (0,\frac 12)$, with $c$ replaced with  $c/2$. 
\\
Since $\divergence u=\trace (\eps u)$, we have that $\abs{\divergence u}^2 \leq n\abs{\eps u}^2$. This inequality and Young's inequality tell us that
        \begin{align}
            \label{snow9}
            \begin{aligned}
                \bigg|\theta \nabla \abs{\eps u}(\nabla^T\divergence u -\Delta u)\frac{\divergence u}{\abs{\eps u}}\bigg| \leq   \tfrac 12 \abs{\Delta u}^2 + \tfrac 12 \abs{\nabla \divergence u}^2  +n\theta^2 \abs{\nabla \abs{\eps u}}^2.
            \end{aligned}
        \end{align}
        One can estimate the left hand side of \eqref{snow80} from below via \eqref{snow9} and obtain:
        \begin{align}\label{2026-8}
                \abs{\nabla \eps u}^2 + \tfrac 12  \abs{\Delta u}^2+ \tfrac 12 \abs{\nabla \divergence u}^2 + 2\theta \abs{\nabla \abs{\eps u}}^2&
                \\ \nonumber
                +\theta \nabla \abs{\eps u}(\nabla^T\divergence u -\Delta u)\frac{\divergence u}{\abs{\eps u}}&\geq \abs{\nabla \eps u}^2 + (2\theta - n\theta^2) \abs{\nabla \abs{\eps u}}^2.
        \end{align}
        If $2\theta - n \theta^2 \geq 0$, then \eqref{snow80} follows with $c=1$. Otherwise,  the inequality $\abs{\nabla \abs{\eps u}}\leq \abs{\nabla \eps u}$ implies that
        \begin{align}\label{2026-9}
         \abs{\nabla \eps u}^2 + (2\theta - n\theta^2) \abs{\nabla \abs{\eps u}}^2   \geq \big(1+ 2\theta - n \theta^2\big)\abs{\nabla \eps u}^2. 
        \end{align}
       Inasmuch as $i_a\leq \theta \leq s_a$,  the assumption \eqref{p_bounds_alldim} ensures that $1+2\theta-n\theta^2>0$. The inequality \eqref{snow80}, with $c=1+2\theta-n\theta^2$, thus follows from \eqref{2026-8} and \eqref{2026-9}.
    \end{proof}

    In dimension $n=2$ the lower bound for the index $i_a$ required in Lemma \ref{lem:pw_ineq_lowdim} can be relaxed.  

    \begin{lemma}[Pointwise inequality for $n=2$]
        \label{lem:pw_ineq_2d}
        Assume that $n=2$.  Let $a:(0,\infty) \to (0,\infty)$ be such 
        the condition \eqref{dec1} is fulfilled and 
        \begin{align}
            \label{p_bounds_2d}
           -\frac{5}{2(4+\sqrt{6})} <i_a \leq  s_a<\infty.   
        \end{align}
        Let   $\mathcal A$ be the function defined by \eqref{mathcalA}. 
        Then the inequality \eqref{pwineq1} holds for every $u\in C^3(\Omega)$, and for some positive constant $c$ depending on $i_a$ and $s_a$.
    \end{lemma}

    \begin{proof}
        Arguing along the same lines as in the  proof of Lemma \ref{lem:pw_ineq_lowdim}, one is reduced to showing that
        \begin{align}\label{2026-10}
            \abs{\nabla \eps u}^2 + \tfrac 12 \abs{\Delta u}^2 + 2\theta \abs{\nabla \abs{\eps u}}^2 + \tfrac 12 \nabla \divergence u \Delta u &\geq c\, \abs{\nabla \eps u}^2
        \end{align}
        for some constant $c>0$ depending only on $i_a$ and $s_a$. Indeed, the inequality \eqref{2026-10} is nothing but \eqref{snow6} with $\delta =0$. Thanks to the inequality 
        $\abs{\Delta u}^2 \leq c(n) \abs{\nabla \eps u}^2$,
        the inequality \eqref{2026-10} implies \eqref{snow6} for some sufficiently small $\delta>0$. 
        \\ If $\theta \geq 0$,  one can argue as in the proof of  
        Lemma of Lemma \ref{lem:pw_ineq_lowdim}. On the other hand, 
        if $\theta <0$, then, since $\abs{\nabla \abs{\eps u}}\leq \abs{\nabla \eps u}$, it suffices to prove that
        \begin{align}
            \label{snow10}
            (1+2\theta) \abs{\nabla \eps u}^2 + \tfrac 12 \abs{\Delta u}^2 +\tfrac 12 \nabla \divergence u \Delta u \geq c \abs{\nabla \eps u}.
        \end{align}
       We claim that
       \begin{align}
            \label{Claim:} (\sqrt{6}-2)\abs{\nabla \eps u}^2+ \abs{\Delta u}^2 +\nabla \divergence u \Delta u \geq 0.
       \end{align}
The inequality \eqref{snow10} will follow from \eqref{Claim:}, owing to the assumption \eqref{p_bounds_2d}.
\\ To prove our claim, let us set
        \begin{align*}
            &\alpha = \partial_{11}u_1, \qquad \beta =\partial_{22}u_1, \qquad\gamma = \partial_{12}u_2, 
            \\
            &\alpha' = \partial_{22}u_2, \qquad\beta' = \partial_{11}u_2, \qquad\gamma' = \partial_{12}u_1.
        \end{align*}
       Thereby,
        \begin{align*}
         \abs{\nabla \eps u}^2 &= \alpha^2 + 2 \left(\frac{\beta'+\gamma'}{2} \right)^2 + \gamma^2 +  \alpha'^2 + 2 \left(\frac{\beta+\gamma}{2} \right)^2 + \gamma'^2
         \\
            \abs{\Delta u}^2 &= (\alpha + \beta)^2 + (\alpha' + \beta')^2,
            \\
            \nabla \divergence u \Delta u&= (\alpha + \gamma)(\alpha + \beta)+(\alpha' + \gamma')(\alpha' + \beta').
        \end{align*}
Set
$$Q=\begin{pmatrix}
                 s+2 & \tfrac 32 & \tfrac 12
                 \\
                 \tfrac 32 & \tfrac s2 + 1 & \tfrac s2 + \tfrac 12
                 \\
                 \tfrac 12 & \tfrac s2 + \tfrac 12 & \tfrac 32 s
             \end{pmatrix},$$
             where  $s=\sqrt{6}-2$. Then, the inequality \eqref{Claim:}
reads:
\begin{align}
    \label{2026-11}
(\alpha, \beta, \gamma) 
             Q (\alpha, \beta, \gamma)^T+ (\alpha', \beta', \gamma')Q (\alpha', \beta', \gamma')^T\geq 0.
\end{align}
This inequality holds for every $(\alpha, \beta, \gamma)$ and $(\alpha', \beta', \gamma')$, since the matrix $Q$ has one eigenvalue equal to zero and two positive eigenvalues, and it is hence positive semidefinite.
\\ The proof is complete.

    \end{proof}

\section{Function spaces}\label{S:functional}
    
Let $B$ be a Young function and let $N\in \mathbb N$. The  Orlicz space $L^B(\Omega)$ is defined as 
\begin{align*}
    L^B (\Omega)= \bigg\{u: \Omega \to \mathbb R^N: u \, \text{is measurable and there exists} \,\, \lambda >0\,\,\text{s.t.} \,\, \int_\Omega B\bigg(\frac{|u|}{\lambda}\bigg)\, d x <\infty\bigg\}.
\end{align*}
It  is a Banach space, equipped with the  Luxemburg norm.
Spaces of matrix-valued functions will also be considered and are defined analogously.
The Orlicz-Sobolev space $W^{1,B}(\Omega)$ is  given by
\begin{align*}
    W^{1,B} (\Omega)= \big\{u \in L^B (\Omega): u \,\,\text{is weakly differentiable and}\,\, \nabla u \in L^B(\Omega)\big\}.
\end{align*}
The local versions $L^B _{\rm loc}(\Omega)$
and $W^{1,B}_{\rm loc} (\Omega)$ of these spaces are defined as usual. Recall that $L^{B_1} _{\rm loc}(\Omega)\subset L^{B_2} _{\rm loc}(\Omega)$ for some Young functions $B_1$ and $B_2$ if and only if $B_1$ dominates $B_2$ near infinity in the sense of Young functions. A parallel characterization holds for the inclusion relation between local Orlicz-Sobolev spaces. 
\\
The space $W^{1,B}_0 (\Omega)$ of those functions in $W^{1,B} (\Omega)$ vanishing on $\partial \Omega$ can be defined as
\begin{multline*}
        W^{1,B}_0 (\Omega)=
        \big\{u \in W^{1,B} (\Omega):  \,\,\text{the extension of $u$ by $0$ to $\mathbb R^n\setminus \Omega$ is weakly differentiable in $\mathbb R^n $}\big\}.    
\end{multline*}
In what follows, functions in the space $W^{1,B}_0 (\Omega)$ will be assumed to be extended as $0$ in $\RRn \setminus \Omega$ without further mention.
\\ Given $u_0\in W^{1,B}(\Omega)$, we also define $W^{1,B}_{u_0}(\Omega)$ as the affine subspace of $W^{1,B}(\Omega)$ of those functions which agree with $u_0$ on $\partial \Omega$ in the sense of Sobolev spaces. Namely, we set
$$W^{1,B}_{u_0}(\Omega)= \{u\in W^{1,B}(\Omega): u-u_0\in W^{1,B}_0 (\Omega)\}.$$
When $B(t)=t^p$ for some $p>1$, the Orlicz and Orlicz-Sobolev spaces defined above agree with standard  Lebesgue and Sobolev type spaces.
\\
A version of the Poincar\'e inequality in Orlicz spaces tells us that, if $\Omega$ has finite Lebesgue measure $|\Omega|$ and the function $B$ is such that $s_B<\infty$, then
\begin{align}
    \label{poincare}
    \int_{\Omega}B(|u|)\, dx \leq c\int_{\Omega}B(|\nabla u|)\, dx
\end{align}
for some $c=c(n,s_B, |\Omega|)$ and every $u \in W^{1,B}_0(\Omega)$. This is a consequence, e.g., of \cite[Lemma 3]{Talenti90}.
\\
The symmetric gradient Orlicz-Sobolev space is defined as
\begin{align}
    \label{sym-orliczsobolev}
    E^{1,B} (\Omega)= \big\{u: \Omega \to \RRn\,: u\in L^B (\Omega)  \,\, \text{and} \,\,\mathcal E u \in L^B(\Omega)\big\}.
\end{align}
Its subspace $ E^{1,B}_0 (\Omega)$ and the space  $ E^{1,B}_{\rm loc} (\Omega)$ can be defined in analogy with $W^{1,B}_0(\Omega)$ and $W^{1,B}_{\rm loc}(\Omega)$, respectively.
 \\ A Korn type inequality in Orlicz spaces asserts that, if $|\Omega|<\infty$   and $1<i_B\leq s_B<\infty$, then 
there exists a constant $c=c(n, i_B, s_B)$  such that
 \begin{align}
    \label{korn-0}\int_{\Omega}B(|\nabla u|)\, dx
    \le
    c\int_{\Omega}B(|\eps u|)\, dx
  \end{align}
  for every $u  \in E^{1,B}_0 (\Omega, \mathbb R^n)$ -- see \cite[Theorem 6.10]{DieningRuzickaSchumacher10} and \cite[Theorem 1]{Fuchs2010}. A variant of the Korn inequality in Orlicz spaces for general Young functions $B$ can be found in \cite{Cianchi14}.
      \\
For $\alpha \in (0,1)$ and $q \in [1,\infty)$, the \Nikolskii~space $\mathcal{N}^\alpha_q(B_R)$ on a ball $B_R$ is defined as
\begin{align*}
\mathcal{N}^\alpha_q(B_R)= \set{ u \in L^q(B_R)\,:\, \norm{u}_{\mathcal{N}^\alpha_q(B_R)}<\infty},
\end{align*}
where
$$\norm{u}_{\mathcal{N}^\alpha_q(B_R)} = \norm{u}_{L^q(B_R)} + \sup_{\abs{h}<R} \norm{\tau_h u}_{L^q(B_{R-\abs{h}})},$$
and $\tau_h u(x)= u(x+h)-u(x)$, see \cite[\S\,18]{BesovIlinNikolskii1978b}.
Note that $\mathcal{N}^\alpha_q(B_R)$ agrees with the Besov space
 $B^{\alpha}_{q, \infty}(B_R)$. Furthermore, as shown in \cite[\S \,26]{BesovIlinNikolskii1978b}, the space $\mathcal{N}^\alpha_q(B_R)$ is compactly embedded into $L^q(B_R)$.
 \\
The following lemma provides us with 
\Nikolskii~regularity of matrix-valued functions $U$ with $\mathcal A(U)$ in a Sobolev space.

\begin{lemma}
    \label{lem:Atou}
    Assume that the function $a\colon (0,\infty) \to (0,\infty)$ is such that $a\in C^1((0,\infty ))$ and $-\frac 12 <i_{a}\leq s_a<\infty$, and let $\mathcal{A}$ be the function associated with $a$ via \eqref{mathcalA}.
    Set
    $\alpha = \min \set{1,\frac{1}{1+s_a}}$. If the function $U\colon B_R \to \RRnn$ is such that $\mathcal{A}(U)\in W^{1,2}(B_R)$ for some ball $B_R\subset \RRn$, then $U\in \mathcal{N}^\alpha_1(B_R)$. 
    Moreover,
\begin{align}
    \label{nikolski-emb}
    \norm{U}_{\mathcal{N}^\alpha_1(B_R)}\leq c\big( \norm{\mathcal A (U)}_{W^{1,2}(B_R)}^2 +1 \big)
\end{align}
for some constant $c$ depending on $n, i_a,s_a, a(1)$, and $R$.
\end{lemma}
\begin{proof} Throughout this proof,
the explicit constants and the constant in the relation $\lq\lq \lesssim"$ depend only on $n, i_a, s_a, a(1)$ and $R$. Let $B$ be the Young function associated with $a$ by the formula \eqref{B}, and let $E$ be the Young function 
 associated with $a^2$ according to the same rule. Namely,
 $$E (t) = \int\nolimits_0^t a^2(\tau)\tau\,d\tau \qquad \text{for $t\geq 0$}.$$ 
 Since $i_{a^2}=2i_a > -1$, one has that $E$ is a uniformly convex Young function in the sense of \cite[Appendix B]{DieningFornasierWank20}, with lower and upper index of uniform convexity $q^-=2i_a +2$ and $q^+=2s_a +2$, respectively. For each $x\in B_R$, define the  \emph{shifted Young functions} $B_{\abs{U(x)}}$ and $E_{\abs{U(x)}}$ as
    \begin{align*}
        B_{\abs{U(x)}}(t)=\int\nolimits_0^t a(\max\set{\abs{U(x)},\tau})\tau \,d\tau\quad\text{and}\quad E_{\abs{U(x)}}(t)=\int\nolimits_0^t a(\max\set{\abs{U(x)},\tau})^2\tau \,d\tau \quad \text{for $t \geq 0$,}    
    \end{align*}
    in the spirit of  \cite[(B.4)]{DieningFornasierWank20}. The indices of uniform convexity of $E_{\abs{U(x)}}$ agree with $\min \set{q^-,2}$ and $\max \set{q^+,2}$.
    Let $x \in B_R$ and $h\in \RRn$ be such that $x+h\in B_R$.  Since $\alpha\max \set{q^+,2}=2$, from 
    Young's inequality one can infer  that
    \begin{align*}
        \abs{h}^{-\alpha}\abs{\tau _h U(x)} &\le E _{\abs{U(x)}}(\abs{h}^{-\alpha} \abs{\tau _h U(x)}) +\widetilde{E_{\abs{U(x)}} } (1) \lesssim \abs{h}^{-2}E _{\abs{U(x)}}(\abs{\tau _h U(x)}) +\widetilde{E_{\abs{U(x)}} } (1).
    \end{align*}
     We have that $\widetilde{E_{\abs{U(x)}} } (1) \lesssim E (\abs{U(x)})+\widetilde{E}(1)\leq \abs{\mathcal{A}(U(x))}^2$, where the first inequality holds by  \cite[Corollary 44]{DieningFornasierWank20}. Moreover, $E_{\abs{U(x)}} (t)\leq E_{\abs{U(x)}}' (t)t = B_{\abs{U(x)}}'(t)^2$ for $t \geq 0$, and, by \cite[Lemma 41]{DieningFornasierWank20},  $B_{\abs{U(x)}}'(\abs{\tau_h U(x)})\lesssim \abs{\tau_h\mathcal{A}(U(x))}$. Thus,
    \begin{align}
        \label{nikolskii1}
        \abs{h}^{-\alpha}\abs{\tau _h U(x)} &\lesssim \abs{h}^{-2} \abs{\tau_h \mathcal A(U(x))} ^2 + \abs{\mathcal{A}(U(x))}^2 + \widetilde{E}(1).
    \end{align}
   Furthermore,
    \begin{align}\label{nikolskii2}
\widetilde{E}(1) \lesssim \widetilde{E}(a(1)^2)/a(1)^2\leq \widetilde{E}'(a(1)^2)= \widetilde{E}'(E'(1))=1
    \end{align}
    Thanks to the assumption $\mathcal A (G)\in W^{1,2}(B_R)$, combining Equations \eqref{nikolskii1} and \eqref{nikolskii2} implies that $U \in \mathcal{N}^\alpha_1(B_R)$ and that the inequality \eqref{nikolski-emb} holds.
\end{proof}

\section{Solutions to symmetric gradient systems}\label{S:solutions}

Let $p>1$. Assume that the function $a: (0, \infty) \to (0, \infty)$ is such that $\in C^1(0, \infty)$ and $s_a\leq p-2$.  Let $\mathcal A$ be the function associated with $a$ as in \eqref{mathcalA}. Assume that
$f \in L^1_{\rm loc}(\Omega) \cap (W^{1,p}_0(\Omega'))^*$ for every open set $\Omega ' \subset \subset \Omega$.
A function $u \in W^{1,p}_{\rm loc}(\Omega)$ is called a local weak solution to the system 
\begin{equation}\label{localeq}
- {\rm {  div}}( \mathcal A (\eps u) ) = { f}  \quad {\rm in}\,\,\, \Omega\,,
\end{equation}
 if
\begin{equation}\label{weaklocl}
\int _{\Omega'} \mathcal A(\eps u) \cdot \eps  \varphi  \,dx = \int _{\Omega '} {f}\cdot \varphi  \,dx
\end{equation}
for every open set $\Omega ' \subset \subset \Omega$, and every function $\varphi \in W^{1,p}_0(\Omega')$.

When  $f\in L^2_{\rm loc}(\Omega)$, standard weak solutions to the system \eqref{eq:system} and, more generally, to \eqref{weaklocl}, are therefore well defined only if $p\geq \frac{2n}{n+2}$.  Solutions in a suitable extended sense have thus to be
considered. For scalar $p$-Laplacian type equations, various (a posteriori equivalent) definitions of solutions, which also suit the case where $f\in L^1_{\rm loc}(\Omega)$, or even where $f$ is just a Radon measure, have been introduced  in the literature. They include   
entropy solutions \cite{BBGGPV95}, renormalized solutions \cite{MasoMuratOrsinaPrignet1999},
 and SOLA \cite{Dallaglio96}. These solutions need
not even be weakly differentiable. 

The case of systems is  subtler and has been less investigated. We adopt a notion of solution modeled on the approach of \cite{DolzHungerMueller97}, which well fits our purposes. Roughly speaking, the relevant solutions  are merely approximately differentiable, and are pointwise limits of solutions to approximating problems with smooth right-hand sides.

Recall that a measurable function $u  : \Omega \to \mathbb R^n$ is said to be approximately differentiable at $x \in \Omega$ if there exists a matrix ${\rm ap} \nabla u(x)  \in \mathbb R^{n\times n}$ such that, for every $\varepsilon >0$,
$$\lim_{r \to 0^+} \frac {\big|\{y\in B_r(x): \frac  1r |u(y)-u(x)- {\rm ap} \nabla u(x)  (y-x)|>\varepsilon\}\big|}{r^n} =0.$$
If $u$ is approximately differentiable at almost every point in $\Omega$, then the function ${\rm ap} \nabla u : \Omega \to \mathbb R^{n\times n}$ is measurable. We also set
$${\rm ap}\eps u= \tfrac 12({\rm ap} \nabla u+({\rm ap} \nabla u)^T).$$
If $u\in W^{1,1}_{\loc}(\Omega)$, then $\mathrm{ap} \nabla u = \nabla u$ a.e. in $\Omega$. 

Assume that $f \in L^2_{\rm loc}(\Omega)$. 
An approximately differentiable function $u : \Omega \to \RRn$ is called a local  approximable  solution to the system
\eqref{weaklocl}
if $|{\rm ap}\eps u| \in L^{p-1}_{\rm loc}(\Omega)$,  
and, for every open set $\Omega' \subset \subset \Omega$,
there exists a sequence $\{f_k\}\subset  C^\infty (\Omega)$, with
$f_k \to f$   in $L^2_{\rm loc} (\Omega)$,
and a corresponding sequence of local weak solutions $\{u_k\}$ to the system
\begin{equation}\label{localeqk}
- {\rm {  div}}( \mathcal A(\eps u_k)) = {f}_k  \quad {\rm in}\,\,\, \Omega'\,,
\end{equation}
such that
\begin{equation}\label{approxuloc}
u_k \to u \quad \hbox{and} \quad \eps u_k \to {\rm ap} \eps u \quad \hbox{a.e. in $\Omega'$,}
\end{equation}
and
\begin{equation}\label{approxaloc}
\lim _{k \to \infty} \int_{\Omega'}|\mathcal{A}(\eps u_k)| \, dx =\int_{\Omega'}|\mathcal{A}(\eps u)|\, dx.
\end{equation}
 In what follows, we shall denote ${\rm ap}\eps u$ simply by $\eps u$.

As already hinted in Section \ref{S:intro}, the notion of approximable solution to the system \eqref{eq:system} with $f\in L^2_{\rm loc}(\Omega)$ is a generalization of that of weak solution. 
This claim is the subject of the following result. We only provide a sketch of its proof, which rests upon  a quite standard argument.
 \begin{proposition}
     \label{weak-approx}
     Assume that $p\geq \frac{2n}{n+2}$ and $f\in L^2_{\rm loc}(\Omega)$. If $u$ is a weak solution to the system \eqref{eq:system}, then it  is also an approximable solution.
 \end{proposition}
 \begin{proof}[Proof, sketched]
    By the Sobolev embedding theorem, the assumption on $p$ ensures that 
    $W^{1,p}_{\rm loc}(\Omega) \subset L^2_{\rm loc}(\Omega)$. It also ensures that, given an open set $\Omega' \subset \subset \Omega$, one has
$L^2(\Omega') \subset (W^{1,p}_0(\Omega'))^*$, with a continuous embedding. 
    Consider any sequence $\{f_k\}\subset C^\infty(\Omega)$ such that $f_k \to f$ in $L^2_{\rm loc}(\Omega)$. Let $\{u_k\}$ be the sequence of weak solutions to the Dirichlet problems
\begin{align}
    		\label{weak-approx1}
    		\begin{alignedat}{2}
    			-\divergence (\mathcal A_p(\eps u_k)) &= f_k &\qquad&\text{in $\Omega'$}
    			\\
    			u_{k}&= u & \qquad&\text{on $\partial \Omega'$.}
                \end{alignedat}
    	\end{align}
        Since $u\in W^{1,p}(\Omega)$, we have that ${\rm ap}\eps u= \eps u$. Thus, it suffices
         to show that, up to subsequences, the sequence $\{u_k\}$ fulfills the properties \eqref{approxuloc} and \eqref{approxaloc}   with ${\rm ap}\eps u$ replaced by $\eps u$. Since $u_k-u\in W^{1,p}_0(\Omega)$, these properties will in turn follow if we prove that
         \begin{align}
             \label{weak-approx9}
             \eps u_k \to \eps u \qquad \text{in $L^p(\Omega')$.}
         \end{align}
        From an use of the test function $u_k-u$ in the weak formulation of the problem \eqref{weak-approx1} and the Sobolev-Poincar\'e inequality  for symmetric gradients one can deduce that
    \begin{align}
            \label{weak-approx3}
      \|\eps u_k\|_{L^{p}(\Omega')}^p &    \leq c
        \|\eps u\|_{L^{p}(\Omega')}^p+\|f\|_{L^{2}(\Omega')}^{p'}
    \end{align}
    for some constant $c$ independent of $k$.
         Using the test function $u_k-u\in W^{1,p}_0(\Omega)$ in the
        problems \eqref{weak-approx1} and \eqref{eq:system},
subtracting the resultant equations, and exploiting  the Sobolev-\Poincare~inequality  for symmetric gradients again yield:
\begin{align}
    \label{weak-approx2}
    \begin{aligned}
         \int_{\Omega'}(\mathcal A_p(\eps u_k)- \mathcal A_p(\eps u))\cdot(\eps u_k - \eps u)\, dx &= \int_{\Omega'}(f_k-f)\cdot(u_k-u)\, dx
 \\
 &\leq c \norm{f_k-f}_{L^2(\Omega')} \big( \norm{\eps u_k}_{L^p(\Omega')} + \norm{\eps u}_{L^p(\Omega')}\big)
    \end{aligned}
\end{align}
for some constant $c$ independent of $k$.
 By our assumption on $f$ we have $f_k \to f$ in $L^2(\Omega')$. Therefore, owing to \eqref{weak-approx3}, the right-hand side of the inequality \eqref{weak-approx2} converges to zero.
Since $\mathcal{A}_p$ is a monotone operator, the convergence \eqref{weak-approx9} follows from \eqref{weak-approx2}.
 \end{proof}

    \section{Nonlinear systems with quadratic growth}
    \label{sec:quadratic}

The pointwise inequalities established in Lemmas \ref{lem:pw_ineq_lowdim} -- \ref{lem:pw_ineq_2d} cannot be directly applied to local solutions to the system \eqref{eq:system}. The reason is twofold. First, the relevant solutions need not belong to the space $C^3(\Omega)$. Second, the assumption \eqref{dec1} is not fulfilled by the function $a_p(t)=t^{p-2}$, if $1<p<2$. To circumvent these obstacles, the solutions  to the system \eqref{eq:system} will be locally approximated by families of solutions to systems with quadratic growth. In the \emph{approximating systems}, the function $a_p$ is replaced with $a_{p,\varepsilon}$, the function defined as in \eqref{aeps} for $\varepsilon \in (0,1)$. 

A critical step in this approximation argument is accomplished in Theorem \ref{thm:quadratic}, the principal result of this section. Theorem \ref{thm:quadratic} deals with a general class of systems with \emph{quadratic growth}, namely systems of the form
    \begin{align}
        \label{system_a}
        -\divergence (\mathcal A (\eps u)) = f \quad \text{in $\Omega$},
    \end{align}
    where $\mathcal A$ is given by \eqref{mathcalA}, with $a:(0,\infty)\to (0, \infty)$ such that the condition \eqref{dec1}  is satisfied and 
    \begin{align}
        \label{assumption_a_nice2}
        \lambda  \leq a(t) \leq \lambda^{-1}  \quad \text{for  $t>0$,}
    \end{align}
and for some  $\lambda \in (0,1)$. Under proper bounds on the indices $i_a$ and $s_a$, 
 an estimate analogous to \eqref{main1}, with $\mathcal A_p$ replaced with $\mathcal A$, is derived for local weak solutions to the system  \eqref{system_a}. The constant appearing in this bound depends on $n$, $i_a$, and $s_a$, but is independent of $\lambda$. This is a key feature in view of its application with $a=a_{p,\varepsilon}$. Indeed,  with such a choice of the function $a$, Equation \eqref{assumption_a_nice2}
holds with $\lambda=\varepsilon$, whereas $i_a$ and $s_a$ are uniformly bounded independently of $\varepsilon$.

In the statement, and in what follows, the notation $\fint$ stands for an averaged integral.

    \begin{theorem}[Systems with quadratic growth]\label{thm:quadratic}
 Let $a\colon (0,\infty ) \to (0,\infty)$ be such that 
the conditions \eqref{dec1} and \eqref{assumption_a_nice2} are fulfilled.
Suppose that one of the following assumptions is holds:
\\ (i) $n=2$ \, and \, $ -\frac{5}{2(4+\sqrt{6})} <i_a \leq s_a <\infty$;
\\ (ii) $3\leq n \leq 7$\, and \, $- \frac{8-n}{16}<i_a \leq s_a <\infty$;
\\ (iii) $n\geq 3$\,  and\, $-\frac{1}{\sqrt{n+1}+1}< i_a \leq s_a< \frac{1}{\sqrt{n+1}-1}$.

 Let $f\in L^2_{\rm loc}(\Omega)$ and let $u$ be a local weak solution to the system \eqref{system_a}. Then, 
\begin{align}
    \label{2026-26}
    \mathcal A(\eps u) \in W^{1,2}_{\rm loc}(\Omega)
\end{align}
        and
    \begin{align}\label{jan147}
        \fint_{B_R}  \abs{ \mathcal A(\eps u)}^2+ R^2\abs{\nabla ( \mathcal A(\eps u))}^2\,dx
            \leq c R^2\fint_{B_{2R}} \abs{f}^2\,dx +c\Bigg(\fint_{B_{2R}} | \mathcal A(\eps u)|\,dx\Bigg)^2 
        \end{align}
         for some constant 
    $c=c(n, i_a, s_a)$ and 
        for every ball $B_R$ such that $B_{2R}\subset\subset \Omega$.
    \end{theorem}

Our proof of Theorem \ref{thm:quadratic}
relies on an integral inequality  deduced from the pointwise estimates exhibited in the previous section.
As a 
 first step, we establish the inequality in question for sufficiently smooth functions. This is the content of Lemma \ref{lem:local_estimate_v}. In the subsequent Lemma \ref{lem:local_estimate_quadratic}, the same integral inequality is extended to functions endowed with locally square integrable second-order weak derivatives. This is enough for our purposes, inasmuch as local weak solutions to  systems with quadratic growth like \eqref{system_a} are known to enjoy this regularity. 
Let us notice that the functions $u$ in Lemmas \ref{lem:local_estimate_v} and \ref{lem:local_estimate_quadratic} are not requested to be solutions of any system.

    \begin{lemma}
        \label{lem:local_estimate_v}
        Assume that the function $a\colon(0,\infty) \to (0,\infty)$ satisfies the same hypotheses as in Theorem \ref{thm:quadratic}. Let $\mathcal A$ be the function defined by \eqref{mathcalA}.
       Let  $u\in C^3(\Omega)$. Then,  
\begin{multline}\label{eq:loca_lem}
            \fint_{B_R}  \abs{ \mathcal A(\eps u)}^2 + R^2\abs{\nabla( \mathcal A(\eps u))}^2\,dx
                \leq c  R^2\fint_{B_{2R}} \abs{\divergence (\mathcal A(\eps u))}^2\,dx +c\Bigg(\fint_{B_{2R}} |\mathcal A(\eps u)|\,dx\Bigg)^2
        \end{multline}
     for some constant 
    $c=c(n, i_a, s_a)$ and for every ball $B_R$ such that $B_{2R}\subset \subset  \Omega$.
    \end{lemma}
    
    \begin{proof} Throughout this proof, the explicit constants and the constants in the relations $\lq\lq \lesssim"$ and $\lq\lq \eqsim"$ depend only on $n$, $i_a$,  and $s_a$.  
         Let $R\leq \sigma <\tau \leq 2R$ and let $\eta \in C_c^\infty(\mathbb R^n)$ be a cut-off function such that 
        \begin{align*}
          {\rm supp}\,\eta \subset B_\tau, \quad  0\leq \eta \leq 1, \quad \eta = 1 \,\, \text{in $B_\sigma$}\quad \text{and}\quad \abs{\nabla \eta}\leq \frac 2{\tau -\sigma}.
        \end{align*}
        Our assumptions on $a$ and $u$ ensure that the hypotheses of  Lemma \ref{lem:pw_ineq_lowdim} or Lemma \ref{lem:pw_ineq_alldim}, or Lemma \ref{lem:pw_ineq_2d} are fulfilled, depending on whether (ii) or (iii), or (i) is in force. 
 Thus, either of the inequalities \eqref{pwineq1} and \eqref{pwineq2} applies. Multiplying the appropriate inequality by $\eta^2$, integrating over $B_{2R}$, and making use of \eqref{nov50} tell us that 
        \begin{align*}
            \int_{B_{2R}}\eta^2\abs{\nabla (\mathcal A(\eps u))}^2 + \eta^2 \divergence X\,dx \lesssim   \int_{B_{2R}} \eta^2 \abs{\divergence (\mathcal A(\eps u))}^2 \,dx ,
        \end{align*}
        where either
        \begin{align}
            \label{Xcase1}
            X=a(\abs{\eps u})^2 ((\Delta u)^T\eps u  - \tfrac 12 \nabla (\abs{\eps u}^2))
        \end{align}
        or
        \begin{align}
            \label{Xcase2}
            X=a(\abs{\eps u})^2 \Big((\Delta u)^T\eps u  - \tfrac 12 \nabla (\abs{\eps u}^2)+\tfrac 12 \divergence u((\Delta u)^T - \nabla \divergence u\Big),
        \end{align}
        depending on whether  
\eqref{pwineq1} or \eqref{pwineq2} is exploited.
        Since $\eta^2 \divergence X=\divergence (\eta^2 X)-\nabla \eta^2 \cdot X$,  the divergence theorem enables one to deduce that
\begin{align}\label{eq:local1}
                \int_{B_{2R}} \eta^2\abs{\nabla (\mathcal A(\eps u))}^2 \,dx \lesssim \int_{B_{2R}}\eta^2 \abs{\divergence (\mathcal A(\eps u))}^2 + \nabla(\eta^2) \cdot X\,dx.
        \end{align}
        If Equation \eqref{Xcase1} holds, then by Young's inequality, \eqref{nov1}, and \eqref{nov50} we have that
        \begin{align*}
                \nabla (\eta^2)\cdot X&\lesssim  a(\abs{\eps u})^2 \eta \abs{\nabla \eta} \big(\abs{\Delta u}\abs{\eps u}+ \abs{\eps u}\abs{\nabla \abs{\eps u}} \big)
                \\
                &\lesssim \delta\eta^2 a(\abs{\eps u})^2 \abs{\nabla \eps u}^2 +\delta^{-1}\abs{\nabla \eta}^2a(\abs{\eps u})^2\abs{\eps u}^2 
                \\
                &\eqsim \delta\eta^2 \abs{\nabla (a(\abs{\eps u}) \eps u)}^2 +\delta^{-1}\abs{\nabla \eta}^2a(\abs{\eps u})^2\abs{\eps u}^2
        \end{align*}
        for $\delta >0$.
        Similarly, if  \eqref{Xcase2} holds, then
        \begin{align*}
            \nabla (\eta^2)\cdot X&\lesssim  a(\abs{\eps u})^2 \eta \abs{\nabla \eta} \big(\abs{\Delta u}\abs{\eps u}+ \abs{\eps u}\abs{\nabla \abs{\eps u}}+\abs{\divergence u}\abs{\Delta u}+\abs{\divergence u} \abs{\nabla \divergence u} \big)
            \\
            &\lesssim \delta\eta^2 a(\abs{\eps u})^2 \abs{\nabla \eps u}^2 +\delta^{-1}\abs{\nabla \eta}^2a(\abs{\eps u})^2\abs{\eps u}^2 
            \\
            &\eqsim \delta\eta^2 \abs{\nabla (a(\abs{\eps u}) \eps u)}^2 +\delta^{-1}\abs{\nabla \eta}^2a(\abs{\eps u})^2\abs{\eps u}^2.
        \end{align*}
Note that the second inequality in the chain above relies upon the fact that $$\abs{\divergence u}=\abs{\trace \eps u}\lesssim \abs{\eps u}.$$
Combining either of the above estimates for the expression $\nabla (\eta^2)\cdot X$  with \eqref{eq:local1}, and choosing $\delta$ small enough enable us to infer that
\begin{align}\label{eq:local2}
            \int_{B_\sigma} \abs{\nabla (\mathcal A(\eps u))}^2 \,dx \lesssim \int_{B_{2R}} \abs{\divergence (\mathcal A(\eps u))}^2 +(\tau -\sigma)^{-2}\int_{B_{\tau}\setminus B_\sigma} |\mathcal A(\eps u)|^2\,dx.
        \end{align}
        A \Poincare type inequality in the form of \cite[Lemma 3.4]{CianchiMazya19}  tells us that
    \begin{multline}\label{nov2}
                \int_{B_\tau \setminus B_\sigma} |\mathcal A(\eps u)|^2\,dx 
                 \lesssim (\tau -\sigma)^2\int_{B_\tau \setminus B_\sigma} \abs{\mathcal A(\eps u)}^2\,dx+ \frac{R}{(\tau - \sigma)^{n+1}} \Bigg(\int_{B_\tau \setminus B_\sigma} |\mathcal A(\eps u)|\,dx\Bigg)^2.  
        \end{multline}
       Coupling \eqref{eq:local2}  with \eqref{nov2} yields:
        \begin{align*}
            \int_{B_\sigma} \abs{\nabla (\mathcal A(\eps u))}^2\,dx \lesssim &\int_{B_{2R}} \abs{\divergence (\mathcal A(\eps u))}^2 +\int_{B_\tau \setminus B_\sigma}  \abs{\nabla (\mathcal A(\eps u))}^2\,dx
            +\frac{R}{(\tau-\sigma)^{n+3}}\Bigg(\int_{B_\tau \setminus B_\sigma} |\mathcal A(\eps u)|\,dx\Bigg)^2.
        \end{align*}
        An application of the 
        \lq\lq hole-filling" argument to the latter inequality ensures that 
        \begin{align*}
            \int_{B_\sigma}\!\abs{\nabla (\mathcal A(\eps u))}^2\,dx \leq &\,\frac{c}{c+1}\int_{B_\tau} \abs{\nabla (\mathcal A(\eps u))}^2\,dx+c'\!\!\int_{B_{2R}} \!\abs{\divergence (\mathcal A(\eps u))}^2
            % \\
            % &
            +
            \frac{c'R}{(\tau\!-\!\sigma)^{n+3}}\Bigg(\int_{B_\tau \setminus B_\sigma} \!\!\!\abs{\mathcal A(\eps u)}\,dx\Bigg)^2
        \end{align*}
        for suitable positive constants $c$ and $c'$.
        Hence, via a  standard 
     iteration lemma (see e.g. \cite[Section 5, Lemma 3.1]{GiaquintaMultipleIntegrals}), we obtain that
\begin{align}\label{eq:local4}
            \int_{B_R} {\abs{\nabla (\mathcal A(\eps u))}^2} \,dx \lesssim  \int_{B_{2R}} \abs{\divergence (\mathcal A(\eps u))}^2 +\frac{1}{R^{n+2}}\Bigg(\int_{B_{2R}} \abs{\mathcal A(\eps u)}\,dx\Bigg)^2.
            \end{align}
        By the mean value \Poincare 's inequality on balls and the triangle inequality, 
\begin{align}\label{eq:local5}
            \int_{B_R} \abs{\mathcal A(\eps u)}^2\,dx \lesssim  R^2\int_{B_{R}} \abs{\nabla (\mathcal A(\eps u))}^2 +\, \frac 1{R^n}\Big(\int_{B_{R}} \abs{\mathcal A(\eps u)}\,dx\Big)^2.
        \end{align}
       The inequality \eqref{eq:loca_lem} follows from \eqref{eq:local4} and \eqref{eq:local5}. 
    \end{proof}
    
    \begin{lemma}
        \label{lem:local_estimate_quadratic}
       Assume that the function $a:(0,\infty) \to (0,\infty)$ satisfies the same hypotheses as in Theorem \ref{thm:quadratic}. Let $\mathcal A$ be the function defined by \eqref{mathcalA}. If $u\in W^{2,2}_{\rm loc}(\Omega)$, then
     \begin{multline}\label{eq:loca_quadratic}
        \fint_{B_R}  \abs{\mathcal A(\eps u)}^2 + R^2\abs{\nabla (\mathcal A(\eps u))}^2\,dx      
                \leq c  R^2\fint_{B_{2R}} \abs{\divergence (\mathcal A(\eps u))}^2\,dx +c\Bigg(\fint_{B_{2R}} \abs{\mathcal A(\eps u)}\,dx\Bigg)^2
        \end{multline}
     for some constant $c=c(n, i_a, s_a)$ and for every ball $B_R$ such that $B_{2R}\subset \subset \Omega$.
    \end{lemma}

    \begin{proof}
        Let  $\{u_k\}\subset C^\infty(\Omega)$  be a sequence of functions such that
        \begin{alignat*}{2}
            u_k&\to u  &&\qquad\text{in $ W^{2,2}_{\rm loc}(\Omega)$},
            \\
            \nabla u_k&\to \nabla u &&\qquad\text{a.e. in $\Omega$,}
            \\
            \nabla ^2 u_k &\to \nabla ^2 u &&\qquad\text{a.e. in $\Omega$.} 
        \end{alignat*}
        Hence, in particular, $\mathcal A( \eps u_k)\to \mathcal A(\eps u)$ and $\nabla (\mathcal A(\eps u_k))\to \nabla (\mathcal A(\eps u))$   a.e. in $\Omega$. Our assumptions on the function $a$ also ensure that
        \begin{align*}
            \abs{\nabla (\mathcal A(\eps u_k))} &= \bigg|a(\abs{\eps u_k})\nabla \eps u_k + a'(\abs{\eps u_k})\frac{\eps u_k}{\abs{\eps u_k}}\nabla \eps u_k \eps u_k\bigg|
            \\
            &\leq (a(\abs{\eps u_k})+a'(\abs{\eps u_k})\abs{\eps u_k})\abs{\nabla \eps u_k}
            \\ &\leq (\lambda+s_a\lambda)\abs{\nabla \eps u_k }
           \lesssim \abs{\nabla \eps u_k}\lesssim |\nabla ^2 u_k| \qquad \text{in $\Omega$.}
        \end{align*}
        Moreover, for every set $\Omega' \subset\subset \Omega$,   there exists a subsequence, still indexed by $k$, and a function $g\in L^2(\Omega')$ such that $|\nabla ^2 u_k|\leq g$ a.e. in $\Omega'$ for  $k \in \mathbb N$.
        Thus, by the dominated convergence theorem, one has  that $\nabla (\mathcal A(\eps u_k)) \to \nabla (\mathcal A(\eps u))$ in $L^2_\loc(\Omega)$. An analogous  argument also tells us that $\mathcal A(\eps u_k)\to \mathcal A(\eps u)$ in $L^2_\loc(\Omega)$.
    \\
        Since $u_k \in C^\infty(\Omega)$,  from Lemma \ref{lem:local_estimate_v} we deduce that
        \begin{multline*}
            \fint_{B_R}  \abs{\mathcal A(\eps u_k)}^2 +  R^2\abs{\nabla (\mathcal A(\eps u_k))}^2\,dx
    \leq c R^2\fint_{B_{2R}} \abs{\divergence (\mathcal A(\eps u_k))}^2\,dx + c \Bigg(\fint_{B_{2R}} \abs{\mathcal A(\eps u_k)}\,dx\Bigg)^2    
        \end{multline*}
        for some constant $c$ depending  on $n$, $i_a$,  and $s_a$.  
        Thanks to the convergences discussed above, passing to the limit in the latter inequality yields \eqref{eq:loca_quadratic}.
    \end{proof}
    We are now ready to prove Theorem \ref{thm:quadratic}.

    \begin{proof}[Proof of Theorem \ref{thm:quadratic}]
    Throughout this proof, the  constant in the relation $\lq\lq \lesssim "$  depends only on $n$, $i_a$,  and $s_a$.  
        Let  $\{f_k\} \subset C^\infty(\Omega)$ be a sequence of functions such that $f_k\to f$ in $L^2_{\rm loc}(\Omega)$ and let $u_k \in W^{1,2}(B_{2R})$ be the solution to the Dirichlet problem
        \begin{align}
    		\label{approx_qudratic}
    		\begin{alignedat}{2}
    			-\divergence (\mathcal A(\eps u_k)) &= f_k &\qquad&\text{in $B_{2R}$}
    			\\
    			u_{k}&= u & \qquad&\text{on $\partial B_{2R}$.}
                \end{alignedat}
    	\end{align}
       Our assumptions on $a$ ensure the existence and uniqueness of such a solution, via the direct method of the Calculus of Variations.
        One has that  $u_k\in W^{2,2}_{\rm loc}(B_{2R})$ -- see, e.g., \cite{Seregin1992} or \cite{GiPa}. Thus, an application of Lemma \ref{lem:local_estimate_quadratic} to $u_k$, with $R$ replaced by $r\in (0,R)$, and the monotone convergence theorem for the Lebesgue integral imply that
        \begin{align}\label{jan140}
      \fint_{B_R}  \abs{\mathcal A(\eps u_k)}^2 + R^2\abs{\nabla (\mathcal A(\eps u_k))}^2\,dx
                    & \lesssim  R^2\fint_{B_{2R}} \abs{f_k}^2\,dx +\Bigg(\fint_{B_{2R}} \abs{\mathcal A(\eps u_k)}\,dx\Bigg)^2 
                \\ \nonumber
                & \lesssim   R^2\fint_{B_{2R}} \abs{f}^2\,dx +\Bigg(\lambda\fint_{B_{2R}} \abs{\nabla u_k}\,dx\Bigg)^2
        \end{align}
        for $k \in \mathbb N$.
\\
        Using the test function $u_k-u \in W_0^{1,2}(B_{2R})$ in the weak formulation of the problem \eqref{approx_qudratic} and   in the definition of local weak solution to \eqref{system_a}, and subtracting the resultant equations yield
        \begin{align}
            \label{jan141b}
        \int_{B_{2R}}(\mathcal A(\eps u_k)-\mathcal A(\eps u)) \cdot (\eps u_k -\eps u)\, dx = \int_{B_{2R}}(f_k-f)(u_k-u)\, dx.
        \end{align}
        Hence, by the Korn, H\"older and Poincar\'e inequalities,
        \begin{align}
            \label{jan141}
                \int_{B_{2R}} |\nabla u_k -\nabla u|^2\, dx & \leq c
        \int_{B_{2R}} |\eps u_k -\eps u|^2\, dx 
        \\ \nonumber
        & \leq c' \bigg(\int_{B_{2R}}|f_k-f|^2\, dx\bigg)^{\frac 12} \bigg(\int_{B_{2R}}|u_k-u|^2\, dx\bigg)^{\frac 12} 
        \\ \nonumber
        & \leq c'' \bigg(\int_{B_{2R}}|f_k-f|^2\, dx\bigg)^{\frac 12} \bigg( \int_{B_{2R}} |\nabla u_k -\nabla u|^2\, dx\bigg)^{\frac 12}
        \end{align}
        for some constant $c$ depending on $n$ and $R$,  and $c'$ and $c''$  also depending on  $\lambda$.
        Dividing through the inequality \eqref{jan141} by 
        $ \big( \int_{B_{2R}} |\nabla u_k -\nabla u|^2\, dx\big)^{\frac 12}$ and passing to the limit as $k\to \infty$ tell us that $\nabla u_k \to \nabla u$ in $L^2(B_{2R})$. Since $u_k-u\in W^{1,2}_0(B_{2R})$, by the \Poincare~inequality we have that $u_k \to u$ in $W^{1,2}(B_{2R})$, and, possibly taking a subsequence, still indexed by $k$, 
         $$\mathcal A(\eps u_k) \to \mathcal A(\eps u)\qquad\text{a.e. in $B_R$.}$$
      Hence, from Equation \eqref{approx_qudratic} we deduce that the sequence $\{\mathcal A(\eps u_k)\}$ is bounded in $W^{1,2}(B_R)$. Thus, there exist a subsequence, still indexed by $k$, and a function $Z\in W^{1,2}(B_R)$ such that
        \begin{alignat*}{2}
           \mathcal A(\eps u_k) &\to Z &&\qquad\text{a.e. in $B_R$,}
            \\
            \mathcal A(\eps u_k) &\to Z &&\qquad\text{in $L^2(B_R)$,}  
            \\
            \mathcal A(\eps u_k) &\weakto Z &&\qquad\text{weakly in $W^{1,2}(B_R)$.}
        \end{alignat*}
       Therefore, $A(\eps u)=Z$, and passing to the limit as $k\to \infty$ in \eqref{jan140} yields \eqref{jan147}. Equation \eqref{2026-26} hence follows, thanks to the arbitrariness of the ball $B_R$.
    \end{proof}

\section{Proof of Theorem \ref{thm:local}}
\label{S:proof-main}

With the material from the previous sections at our disposal, we are now in a position to accomplish the proof of our main result.

\begin{proof}[Proof of Theorem \ref{thm:local}] For ease of presentation, we split the proof into steps. Unless otherwise specified, throughout this proof, the explicit constants and the constants in the relations $\lq\lq \lesssim"$ and $\lq\lq \eqsim"$ depend only on $n$, and $p$.
\smallskip
\par\noindent
\textbf{Step 1:}\emph{ approximating problems.}
Assume  that $f\in C^\infty(\Omega)$ and that $u$ is a local weak solution to the system \eqref{eq:system}. Let $\{u_m\}$ be a family of functions in $C^\infty (\Omega)$ such that \begin{align}\label{dec36}
	    \lim_{m\to \infty} u_m \to u \qquad \text{in $W^{1,p}_{\rm loc}(\Omega)$.}
	\end{align} 
  Given $\varepsilon \in (0,1)$, let 
  $a_{p,\varepsilon}$, $b_{p,\varepsilon}$, $B_{p,\varepsilon}$ and 
    $\mathcal A_{p,\varepsilon}$ be the functions defined in Section \ref{sec:young}.
    Recall from \eqref{indici} that if $p\leq 2$, then 
  \begin{align}
        \label{jan160}
0\geq s_{a_{p,\varepsilon}}  \geq 
        i_{a_{p,\varepsilon}} > \begin{cases}
  - \frac{5}{2(4+\sqrt{6})} &\quad \text{if $n=2$}
   \\ - \frac{1}{\sqrt{n+1}+1}  &\quad \text{if $n\geq 3$}
        \end{cases}
    \end{align} 
    whereas if $p\geq 2$, then \begin{align}\label{jan161}
        0\leq i_{a_{p,\varepsilon}}\leq s_{a_{p,\varepsilon}} < \begin{cases}
     \infty &\quad \text{if $n\leq 7$}
   \\   \frac{1}{\sqrt{n+1}-1} &\quad \text{if $n\geq 8$.}  
   \end{cases}
    \end{align}
Let $u_{\varepsilon,m}\in W^{1,2}_{u_m}(B_{2R})$ be the weak solution to the Dirichlet problem
	\begin{align}
		\label{approx_system}
		\begin{alignedat}{2}
			-\divergence (\mathcal A_{p,\varepsilon}(\eps u_{\varepsilon,m})) &= f &\qquad &\text{in $B_{2R}$}
			\\
		u_{\varepsilon,m} &= u_m &\qquad&\text{on $\partial B_{2R}$.}     
    	\end{alignedat}
	\end{align}
	Such a solution exists since, for fixed $p$ and $\varepsilon$,  we have that $a_{p,\varepsilon}(t) \eqsim 1$ for $t \geq 0$.
   \medskip
\par\noindent \textbf{Step 2:}\emph{ limit as $\varepsilon \to 0$.} Let $f$ and $u$ be as in Step 1.
In this step we  fix $m \in \mathbb N$, and   we show that the sequence $\{u_{\varepsilon , m}\}$ converges, as $\varepsilon \to 0$, to the 
weak solution $w_m \in W^{1, \min\{p,2\}}_{u_m}(B_{2R})$ to the Dirichlet problem
\begin{align}
		\label{limit_system}
		 \begin{alignedat}{2}
			-\divergence (\mathcal A_p(\eps w_m)) &= f &\qquad&\text{in $B_{2R}$}
			\\
			w_m&=u_m &\qquad&\text{on $\partial B_{2R}$.}    \end{alignedat}
\end{align}
We also prove that $\mathcal{A}_p (\eps w_m)\in W^{1,2}(B_R)$, and 
\begin{align}\label{dec32}
		\fint_{B_{R}} \abs{\mathcal A_p (\eps w_m)}^2+R^2\abs{\nabla \mathcal A_p (\eps w_m)}^2\,dx \lesssim \fint_{B_{2R}} R^2\abs{f}^2\,dx + \Bigg(\fint_{B_{2R}}\abs{\mathcal A_p(\eps w_m)}\,dx\Bigg)^2.
	\end{align}
To this purpose, we proceed as follows. Testing the system \eqref{approx_system}
with $u_{\varepsilon,m} - u_m \in W_0^{1,2}(B_{2R})$ yields, via Young's inequality and the $\Delta_2$-property for $B_{p,\varepsilon}$, the Poincar\'e inequality in Orlicz spaces \eqref{poincare}, and the Korn inequality in Orlicz spaces  \eqref{korn-0}:
\begin{align}
    \label{step1A}
        \int_{B_{2R}}\big(\mathcal A_{p,\varepsilon}(\eps u_{\varepsilon,m})\cdot \eps u_{\varepsilon,m} & - \mathcal A_{p,\varepsilon}(\eps u_{\varepsilon,m})\cdot \eps u_m\big) \, dx  = \int_{B_{2R}} f (u_{\varepsilon,m} - u_m)dx
        \\ \nonumber
        & \leq  \int_{B_{2R}} c_\delta\widetilde{B_{p,\varepsilon}}(|f|)+ \delta  {B_{p,\varepsilon}}(|u_{\varepsilon,m} -u_m|)\, dx
        \\ \nonumber
        &\leq  \int_{B_{2R}}c_\delta \widetilde{B_{p,\varepsilon}}(|f|) +c\,\delta  {B_{p,\varepsilon}}(|\eps u_{\varepsilon,m} |) + c\,\delta  {B_{p,\varepsilon}}(|\eps u_m |)\,dx 
\end{align}
for $\delta >0$ and some constant $c_\delta$ depending  on $n$, $p$ $R$, and $\delta$. Notice that, since the $\Delta_2$-property for $B_{p,\varepsilon}$ holds with a constant independent of $\varepsilon$, the constants in the Poincar\'e and Korn inequalities are also independent of $\varepsilon$.
From \eqref{mar43}, \eqref{step1A}, Young's inequality, and \eqref{orliczbasic}
 we obtain that
\begin{align*}
    \begin{aligned}
        \int_{B_{2R}}B_{p,\varepsilon}(\abs{\eps u_{\varepsilon,m}})\, dx &\lesssim \int_{B_{2R}}\mathcal A_{p,\varepsilon}(\eps u_{\varepsilon,m})\cdot \eps u_{\varepsilon,m} \,dx
        \\
        &\leq\int_{B_{2R}}  A_{p,\varepsilon}(\eps u_{\varepsilon,m})\cdot \eps u_m + c_\delta \widetilde{B_{p,\varepsilon}}(|f|) +\delta  {B_{p,\varepsilon}}(|\eps u_{\varepsilon,m} |) + \delta  {B_{p,\varepsilon}}(|\eps u_m |)\,dx
        \\
        &\leq \int_{B_{2R}}c_\delta \widetilde{B_{p,\varepsilon}}(|f|) +\delta  {B_{p,\varepsilon}}(|\eps u_{\varepsilon,m} |) + c_\delta'  {B_{p,\varepsilon}}(|\eps u_m |)\,dx
    \end{aligned}
\end{align*}
for suitable constants $c_\delta$ and $c_\delta'$ depending on $n$, $p$, $R$, and $\delta$. Choosing 
$\delta$ small enough, we conclude from the latter chain that
\begin{align}
    \label{step1B}
    \int_{B_{2R}}B_{p,\varepsilon}(\abs{\eps u_{\varepsilon,m}})\, dx \leq c \int_{B_{2R}}\widetilde{B_{p,\varepsilon}}(|f|) + {B_{p,\varepsilon}}(|\eps u_m |)\,dx
\end{align}
for some  constant $c$  depending on $R$, $p$, and $n$. Coupling \eqref{step1B} with the Korn inequality \eqref{korn-0} ensures that 
\begin{align}
    \label{step1C}
    \begin{aligned}
        \int_{B_{2R}} \!\! B_{p,\varepsilon}(|\nabla (u_{\varepsilon,m} \!-\!  u_m)|)\,dx & \leq c 
    \!\!\int_{B_{2R}} \!\! B_{p,\varepsilon}(|\eps (u_{\varepsilon,m} \!-\!  u_m)|)\,dx  
 \leq c'\!\!\int_{B_{2R}} \!\!\widetilde{B_{p,\varepsilon}}(|f|) + {B_{p,\varepsilon}}(|\eps u_m |)\,dx
    \end{aligned}
\end{align}
for some  constants $c$ and $c'$ depending on $n$, $p$, and $R$.
Since $f \in L^\infty(B_{2R})$ and $u_m\in W^{1,\max\{2,p\}
}(B_{2R})$, Equations \eqref{step1C}, \eqref{mar45}, and \eqref{mar45'} imply that
\begin{align}
    \label{mar47}
    \int_{B_{2R}} {B_{p,\varepsilon}}(|\nabla u_{\varepsilon,m}-\nabla u_m|)dx \leq c
\end{align}
for some constant $c$ independent of $\varepsilon$. From \eqref{nov41} we infer that the space $W^{1,B_{p,\varepsilon}}_0(B_{2R})$ is continuosly embedded into $ W^{1,\min\{p,2\}}_0(B_{2R})$, with embedding constant independent of $\varepsilon$. Thus, the sequence $\{u_{\varepsilon,m}-u_m\}$ is uniformly bounded in $W^{1,\min\{p,2\}}_0(B_{2R})$ for $\varepsilon \in (0,1)$. This shows that there exists $w_m \in W^{1,\min\{p,2\}}_{u_m}(B_{2R})$ such that 
\begin{align}
    \label{convergence1}
    u_{\varepsilon,m} \weakto w_m \qquad\text{weakly in $W^{1,\min\{p,2\}}_{u_m} (B_{2R})$}
\end{align}
as $\varepsilon \to 0$.
By Theorem \ref{thm:quadratic}, for every ball $B_r(x_0)$ with $B_{2r}(x_0)\subset \subset B_{2R}$ one has that
\begin{align}
    \label{step1D}
    \int_{B_R}  \abs{ \mathcal{A}_{p,\varepsilon} (\eps u_{\varepsilon,m})}^2+ \abs{\nabla (\mathcal{A}_{p,\varepsilon} (\eps u_{\varepsilon,m}))}^2\,dx
            \leq c  \int_{B_{2R}} \abs{f}^2\,dx +c\Bigg(\int_{B_{2R}} | \mathcal{A}_{p,\varepsilon} (\eps u_{\varepsilon,m})|\,dx\Bigg)^2 
\end{align}
for some constant $c$ depending on  $n$, $p$, and $r$. Owing to Young's inequality and the property \eqref{orliczbasic}, 
\begin{align*}
    \abs{\mathcal{A}_{p,\varepsilon}(\eps u_{\varepsilon,m})} \leq c B_{p,\varepsilon} (\eps u_{\varepsilon,m}) + c B_{p,\varepsilon} (1)
\end{align*}
for some constant $c$ independent of $\varepsilon$. Hence, the bound \eqref{step1D} entails that
\begin{align}\label{2026-30}
		\norm{\mathcal A_{p,\varepsilon}(\eps u_{\varepsilon,m})}_{W^{1,2}(B_r(x_0))} \leq c
\end{align}
for some constant $c$ depending on $f$, $r$, $n$, and $p$, but independent of $\varepsilon$ and $m$. Thus, for every such ball, there exists $\overline{\mathcal A}_{r,x_0}\in W^{1,2}  (B_r(x_0))$ such that, up to subsequences,
	\begin{align*}
    \begin{alignedat}{2}
		\mathcal A_{p,\varepsilon} (\eps u_{\varepsilon,m}) &\to \overline{\mathcal A}_{r,x_0} &\quad&
        \text{a.e. in $B_r(x_0)$,}
         \\  
		\mathcal A_{p,\varepsilon} (\eps u_{\varepsilon,m}) &\to \overline{\mathcal  A}_{r,x_0} && \text{in $L^2(B_r(x_0))$,  }
        \\ 
    \mathcal A_{p,\varepsilon} (\eps u_{\varepsilon,m}) &\weakto \overline{\mathcal  A}_{r,x_0} &&\text{weakly in $W^{1,2}(B_r(x_0))$}
    \end{alignedat}
    \end{align*}
    as $\varepsilon \to 0$.
	Therefore, by a covering and diagonal argument, there exists a function $\overline {\mathcal A} \in W^{1,2}_\loc (B_{2R})$ such that
	\begin{align}
    \label{nov45}
		\begin{cases}\mathcal A_{p,\varepsilon} (\eps u_{\varepsilon,m}) \to \overline{\mathcal A}&\quad
        \text{a.e. in $B_{2R}$}
         \\  
		\mathcal A_{p,\varepsilon} (\eps u_{\varepsilon,m}) \to \overline{\mathcal  A} & \quad \text{in $L^2_\loc(B_{2R})$  }
         \\ 
         \mathcal A_{p,\varepsilon} (\eps u_{\varepsilon,m}) \weakto \overline{\mathcal  A} &\quad \text{weakly in $W^{1,2}_\loc(B_{2R})$}
         \end{cases}
    \end{align}
    as $\varepsilon \to 0$. 
      From the bound \eqref{2026-30} and Lemma \ref{lem:Atou} we infer that there exists $\alpha >0$ such that the sequence $\set{\eps u_{\varepsilon,m}}$ is uniformly bounded in the \Nikolskii~space $\mathcal{N}_1^\alpha(B_r(x_0))$ for $\varepsilon \in (0,1)$. Thanks to the compactness of the embedding of $\mathcal{N}_1^\alpha(B_r(x_0))$ into  $L^1(B_r(x_0))$ and  a covering and diagonal argument we conclude that, up to subsequences, 
   \begin{align}
		\label{convergence1A}
		\eps u_{\varepsilon,m} \to \eps w_m \qquad\text{a.e. in $B_{2R}$,}
	\end{align}
    as $\varepsilon \to 0$. 
   By \eqref{conva}, the map $(\varepsilon, \xi)\mapsto \mathcal A_{p,\varepsilon}^{-1}(\xi)$ is continuous. Thus, the convergence \eqref{nov45} implies that
	\begin{align}
		\eps u_{\varepsilon,m} = \mathcal A_{p,\varepsilon}^{-1}(\mathcal A_{p,\varepsilon}(\eps u_{\varepsilon,m})) \to \mathcal A_p^{-1}(\overline{\mathcal A})\qquad\text{a.e. in $B_{2R}$,}
	\end{align}
    as $\varepsilon \to 0$.
	Coupling this piece of information with \eqref{convergence1A} tells us that $\overline{\mathcal A}=\mathcal A_p(\eps w_m)$.
    \\
    Let $\phi \in C_c^\infty (B_{2R})$. Then,
	\begin{align}\label{nov46}
		\int_{B_{2R}} \mathcal A_{p,\varepsilon}(\eps u_\varepsilon)\cdot\eps \phi\,dx=\int_{B_{2R}}f\cdot \phi\,dx.
	\end{align}
	By \eqref{nov45}, one has $\mathcal A_{p,\varepsilon}(\eps u_{\varepsilon,m})\to \mathcal A_p (\eps w_m)$ in $L^2(\support \phi)$. Hence, passing to the limit as $\varepsilon \to 0$ in 
    \eqref{nov46} yields:
	\begin{align}
		\label{convergence0}
		\int_{B_{2R}}\mathcal A_p(\eps w_m)\cdot\eps \phi\,dx = \int_{B_{2R}}f\cdot \phi \,dx
	\end{align}
   for every $\phi \in C_c^\infty(B_{2R})$.
   	Notice that, since $B_{p,\varepsilon}(\abs{\eps u_\varepsilon})\to \abs{\eps w_m}^p$ a.e. in $B_{2R}$ as $\varepsilon \to 0$, by Fatou's Lemma and \eqref{mar47} we have
	\begin{align*}
		\int_{B_{2R}}\abs{\eps w_m}^p \,dx &= \int_{B_{2R}} \liminf_{\varepsilon\to 0} B_{p,\varepsilon}(\abs{\eps u_\varepsilon}) \,dx
		\leq \liminf_{\varepsilon\to 0}\int_{B_{2R}}B_{p,\varepsilon}(\abs{\eps u_\varepsilon})\,dx \leq c.
	\end{align*}
Hence, $\mathcal A_p(\eps w_m)\in L^{p'}(B_{2R})$. An approximation argument then ensures that \eqref{nov46} also holds for every $\phi \in W_0^{1,p}(B_{2R})$.
	This shows that $w_m$ is a weak solution to the problem \eqref{limit_system}.
\\ Next, let $0<r<R$. By Theorem \ref{thm:quadratic},  for each $\varepsilon\in (0,1)$ we have:
	\begin{align*}
		\fint_{B_r} \abs{\mathcal A_{p,\varepsilon} (\eps u_\varepsilon)}^2+r^2\abs{\nabla (\mathcal A_{p,\varepsilon} (\eps u_\varepsilon))}^2\,dx \lesssim r^2 \fint_{B_{2r}} \abs{f}^2\,dx + \Bigg(\fint_{B_{2r}}\abs{\mathcal A_{p,\varepsilon}(\eps u_\varepsilon)}\,dx\Bigg)^2, 
	\end{align*}
   up to constants which, owing to \eqref{jan160} and \eqref{jan161}, are independent of $\varepsilon$ and $r$. Since $\overline {\mathcal A}= \mathcal A_p(\eps w_m)$, from \eqref{nov45} we deduce that
	\begin{align*}
		\fint_{B_{r}} \abs{\mathcal A_p (\eps w_m)}^2+r^2\abs{\nabla (\mathcal A_p (\eps w_m))}^2\,dx \lesssim r^2 \fint_{B_{2r}} \abs{f}^2\,dx + \bigg(\fint_{B_{2r}}\abs{\mathcal A_p(\eps w_m)}\,dx\bigg)^2.
	\end{align*}
	Taking the limit as $r\to R$ in the last inequality yields \eqref{dec32}, via the monotone convergence theorem.
    
     \medskip
\par\noindent
     \textbf{Step 3:}\emph{ limit as $m\to \infty$.} Let $f$ and $u$ be as in Steps 1 and 2. This step is devoted to showing that
    \begin{align}
        \label{dec33}
        \lim_{m \to \infty} w_m \to u \qquad \text{in $W^{1,p}(B_{2R})$.}
    \end{align}
    We will also prove that $\mathcal A_p(\eps u) \in W^{1,2}(B_R)$ and
the inequality  \eqref{main1} holds.
\\ To this end, one can test the problem \eqref{limit_system} with $w_m - u_m$ and proceed along the same lines as in the proof of the inequality \eqref{mar47} to obtain  that
\begin{align}
    \label{dec35'}
    \|w_m - u_m\|_{W^{1,p}_0(B_{2R})}\leq c ,
\end{align}
for some constant $c$. Of course, now the function  $\mathcal A_{p,\varepsilon}(\xi)$ has to be replaced with $\mathcal A_p(\xi)$, and $B_{p,\varepsilon}(t)$ with $t^p$.
 From  \eqref{dec36} and \eqref{dec35'} we deduce that
 \begin{align}
    \label{dec35}
\|w_m\|_{W^{1,p}(B_{2R})}\leq c
\end{align}
for some constant $c$. 
\\ Next, test the system in \eqref{limit_system} and the system \eqref{eq:system} with the function $w_m - u_m \in W^{1,p}_0(B_{2R})$, and subtract the resultant equations. This yields:
\begin{align}
    \label{dec37}
    \begin{aligned}
  0 &=  \int_{B_{2R}}(\mathcal A_p(\eps w_m) - \mathcal A_p(\eps u))\cdot (\eps w_m- \eps u_m) dx
  \\
  & = \int_{B_{2R}}(\mathcal A_p(\eps w_m) - \mathcal A_p(\eps u))\cdot (\eps w_m- \eps u) dx
   +  \int_{B_{2R}}(\mathcal A_p(\eps w_m) - \mathcal A_p(\eps u))\cdot (\eps u- \eps u_m) dx.
   \end{aligned}
\end{align}
Thanks to Equations \eqref{dec36} and \eqref{dec35}, and H\"older's inequality
\begin{align}
    \label{dec38}
    \lim_{n \to \infty}\int_{B_{2R}}(\mathcal A_p(\eps w_m) - \mathcal A_p(\eps u))\cdot (\eps u- \eps u_m) dx=0.
\end{align}
On the other hand, there exists a positive constant $c$ such that
\begin{align}
    \label{dec39}
    (A(\eps w_m) - A(\eps u))\cdot (\eps w_m- \eps u) \geq c (|\eps w_m|+ |\eps u|)^{p-2}|\eps w_m - \eps u|^2.
\end{align}
Thus, if $p\geq 2$, then
\begin{align}
    \label{dec40}
    \int_{B_{2R}}|\eps w_m-\eps u|^p dx & \leq \int_{B_{2R}} (|\eps w_m|+ |\eps u|)^{p-2}|\eps w_m - \eps u|^2dx 
    \\ \nonumber & \leq \frac 1c\int_{B_{2R}}(\mathcal A_p(\eps w_m) - \mathcal A_p(\eps u))\cdot (\eps w_m- \eps u) dx.
\end{align}
If $p<2$, then, by H\"older's inequality,
\begin{align}
    \label{dec41}
    \int_{B_{2R}}|\eps w_m-\eps u|^p dx & \leq \bigg(\int_{B_{2R}} (|\eps w_m|+ |\eps u|)^{p-2}|\eps w_m - \eps u|^2dx \bigg)^{\frac{p}{2}}
\\ \nonumber & \qquad \times \bigg(\int_{B_{2R}} (|\eps w_m|+ |\eps u|)^{p}dx \bigg)^{\frac{2-p}{2}}
    \\ \nonumber & \leq \bigg(\frac 1c\int_{B_{2R}}(\mathcal A_p(\eps w_m) - \mathcal A_p(\eps u))\cdot (\eps w_m- \eps u) dx\bigg)^{\frac{p}{2}}
    \\ \nonumber & \qquad \times \bigg(\int_{B_{2R}} (|\eps w_m|+ |\eps u|)^{p}dx \bigg)^{\frac{2-p}{2}}.
\end{align}
From \eqref{dec37}--\eqref{dec41} one infers that 
\begin{align}
    \label{dec42}
    \lim_{m \to \infty}\|\eps w_m - \eps u\|_{L^p(B_{2R})}=0.
\end{align}
Notice that, for $p<2$, the bound \eqref{dec35} plays a role. By Korn's inequality, 
\begin{align}
    \label{dec43}
    \|\nabla w_m - \nabla u\|_{L^p(B_{2R})} & \leq  \|\nabla w_m - \nabla u_  m\|_{L^p(B_{2R})} +  \|\nabla u_m - \nabla u\|_{L^p(B_{2R})}
    \\ \nonumber & 
    \lesssim  \|\eps w_m - \eps u_m\|_{L^p(B_{2R})} +  \|\nabla u_m - \nabla u\|_{L^p(B_{2R})}  \\ \nonumber & 
    \leq \|\eps w_m - \eps u\|_{L^p(B_{2R})} 
    + \|\eps u - \eps u_m\|_{L^p(B_{2R})} +  \|\nabla u_m - \nabla u\|_{L^p(B_{2R})} 
\end{align}
and, by the Poincar\'e inequality,
\begin{align}
    \label{dec44}
    \|w_m -  u\|_{L^p(B_{2R})} & \leq  \| w_m -  u_m\|_{L^p(B_{2R})} +  \| u_m - u\|_{L^p(B_{2R})}
    \\ \nonumber & 
    \lesssim  \|\nabla w_m - \nabla u_  m\|_{L^p(B_{2R})} +  \|u_m -  u\|_{L^p(B_{2R})}
     \\ \nonumber & 
    \leq  \|\nabla w_m - \nabla u\|_{L^p(B_{2R})} 
    +\|\nabla u - \nabla u_m\|_{L^p(B_{2R})}
    +  \|u_m -  u\|_{L^p(B_{2R})}
    .
\end{align}
Combining \eqref{dec42}--\eqref{dec44} with \eqref{dec35} implies \eqref{dec33}.
\\ The inequalities \eqref{dec32} and \eqref{dec35} enable us to deduce that
\begin{align}
    \label{dec45}
    \|\mathcal A_p(\eps w_m)\|_{W^{1,2}(B_R)}\leq c
\end{align}
for some constant $c$. Thereby, there exists a subsequence, still indexed by $m$, and a function $U\in W^{1,2}(B_R)$ such that    
\begin{align}
    \label{dec46}
    \eps w_{m} \to \eps u \qquad \text{a.e. in $B_{2R}$,}
\end{align}
and 
\begin{align}
    \label{dec47}
\mathcal A_p(\eps w_m) \to U \,\, \text{in $L^2(B_R)$} \quad  \text{and} \quad   \mathcal A_p(\eps w_m) \weakto U \,\, \text{weakly in $W^{1,2}(B_R)$.}
\end{align}
The limit \eqref{dec46} and the first limit in \eqref{dec47} show $U= \mathcal A_p(\eps u)$. Hence $\mathcal A_p(\eps u) \in W^{1,2}(B_R)$, and 
\begin{align}
\mathcal A_p(\eps w_m) \weakto \mathcal A_p(\eps u) \quad \text{in $W^{1,2}(B_R)$.}
\end{align}
Altogether, the inequality \eqref{main1}
follows by passing to the limit as $m \to \infty$ in \eqref{dec32}.

\medskip
\par\noindent
 \textbf{Step 4:}\emph{ Conclusion.} By Steps 1--3, our proof is  complete in the case when $f\in C^\infty(\Omega)$
 and $u$ is a local weak solution.  These assumptions are removed in the present step.
 \\
 Let $f\in L^2_{\rm loc}(\Omega)$ and let $u$ be an approximable solution to the system $\eqref{eq:system}$. Fix any ball $B_{2R}\subset \subset  \Omega$, and let $\Omega'$ be an open set such that $B_{2R}\subset \subset \Omega' \subset \subset \Omega$.
 Let $\{f_k\}$ be a sequence of functions in  $C^\infty (\Omega)$ 
 and 
 $\{u_k\}$ a sequence of weak solutions to the systems \eqref{localeqk}
as in the definition of local approximable solution to \eqref{eq:system}. An application of the result established in Step 3, with $f$ and $u$ replaced with $f_k$ and $u_k$, ensures that $\mathcal A_p(\eps u_k) \in W^{1,2}_{\rm loc}(\Omega')$ and 
\begin{align}\label{main1k}
 R^{-1}\|\mathcal A_p(\eps u_k)\|_{L^2(B_R)}  + \,\|\nabla (\mathcal A_p(\eps u_k))\|_{L^2(B_R)}
   & \leq c \big(\,\|{f_k}\|_{L^2(B_{2R})} + R^{-\frac n2-1}\|\mathcal A_p(\eps u_k)\|_{L^{1}(B_{2R})}\big)
   \\ \nonumber & = c \big(\,\|{f_k}\|_{L^2(B_{2R})} + R^{-\frac n2-1}\|\eps u_k\|_{L^{p-1}(B_{2R})}\big)
\end{align} 
fro some constant $c$ and for every ball $B_R$ such that $B_{2R}\subset \subset  \Omega$. Thanks to Equation \eqref{approxaloc}, the sequence $\{\mathcal A_p(\eps u_k)\}$ is bounded in $W^{1,2}(B_R)$. Thus, there exists 
  a subsequence, still indexed by $k$,  and a function $Z\in W^{1,2}(B_R)$ such that  
\begin{align}
    \label{dec48}
    \eps u_k \to \eps u \qquad \text{a.e. in $B_{2R}$,}
\end{align}
and 
\begin{align}
    \label{dec49}
\mathcal A_p(\eps u_k) \to Z \,\, \text{in $L^2(B_R)$} \quad  \text{and} \quad   \mathcal A_p(\eps u_k) \weakto Z \,\, \text{weakly in $W^{1,2}(B_R)$.}
\end{align}
By the limit \eqref{dec48} and the first limit in \eqref{dec49}, one has  that $Z= \mathcal A_p(\eps u)$. Consequently,  $\mathcal A_p(\eps u) \in W^{1,2}(B_R)$, and 
\begin{align}
\mathcal A_p(\eps u_k) \weakto \mathcal A_p(\eps u) \,\, \text{weakly in $W^{1,2}(B_R)$.}
\end{align}
Taking the limit in \eqref{main1k} as $k\to \infty$ thus yields the inequality
 \eqref{main1}.
\end{proof}

\begin{remark}
    \label{rem:orlicz2}  {\rm
    The result of our main Theorem~\ref{thm:local} admits extensions which include  a broader class of models for the law $\mathcal{A}(\eps u)$, with $\mathcal A$ defined as in \eqref{mathcalA}. Most of the arguments presented above for $\mathcal A=\mathcal A_p$ can be generalized without substantial changes under the assumption that  either $s_a \le 0$ or  $i_a \ge 0$, corresponding to the sub-quadratic and the super-quadratic regime.
    In the case of the power type law $\mathcal A_p$, such an assumption is exploited in some steps of our proofs, and amounts to requiring that  $p\leq 2 $ or $p\geq 2$, respectively. Additional difficulties arise in the situation when  $i_a < 0 < s_a$. Theorem~\ref{thm:quadratic}  provides an example in this connection. Indeed, if the indices $i_a$ and $s_a$ have opposite signs,  they are requested to satisfy stronger assumptions for $3\leq n \leq 7$ than the corresponding ones for $p$ in Theorem \ref{thm:local}.
    Improving these conclusions in this generality entails  further investigations.}
\end{remark}

\par
\bigskip
\noindent

 \par\noindent {\bf Data availability statement.} Data sharing is not applicable to this article as no datasets were generated or analyzed during the current study.

\section*{Compliance with Ethical Standards}\label{conflicts}

\smallskip
\par\noindent
{\bf Funding}. This research was partly funded by:
\\  (i) IRTG 2235 of the German Research Foundation, grant number 
282638148 (L.Behn and L.Diening);
\\ (ii) GNAMPA   of the Italian INdAM - National Institute of High Mathematics (grant number not available)  (A.Cianchi);
\\ (iii) Research Project   of the Italian Ministry of Education, University and
Research (MIUR) Prin 2022 ``Partial differential equations and related geometric-functional inequalities'',
grant number 20229M52AS, cofunded by PNRR (A.Cianchi).
\\ (iv) National Key R\&D Program of China, grant number 2025YFA1018400 (F. Peng)
\\ (v) National Natural Science Foundation of China, grant number 12571209 (F. Peng).

\bigskip
\par\noindent
{\bf Conflict of Interest}. The authors declare that they have no conflict of interest.

\printbibliography
\end{document}